\documentclass[12pt,leqno]{article}
\usepackage{amssymb,amsfonts,amsmath,amsthm,amscd,mathrsfs}
\def\ovxkpoint{\ov{X}_\point{\hspace{-0.7ex}^k}}
\def\wtxkpoint{\wt{X}_\point{\hspace{-0.7ex}^k}}

\def\Xdotwtb{X_{\cdot \wedge T_{\wt{B}}}}

\def\IP{{\mathbb P}}
\def\IR{{\mathbb R}}
\def\IN{{\mathbb N}}
\def\IT{{\mathbb T}}

\def\IZ{{\mathbb Z}}

\def\n{\noindent}
\def\dsl{\textstyle\sum\limits}

\def\dis{\displaystyle}
\def\o{\omega}
\def\fr{\mbox{\footnotesize $\dis\frac{1}{2}$}}
\def\frvier{\mbox{\footnotesize $\dis\frac{1}{4}$}}
\def\ov{\overline}
\def\ve{\varepsilon}
\def\f{\footnotesize}
\def\r{\rightarrow}
\def\point{{\mbox{\large $.$}}}
\def\wh{\widehat}
\def\wt{\widetilde}

\def\cA{{\cal A}}
\def\cB{{\cal B}}

\def\cT{{\cal T}}

\def\cI{{\cal I}}

\def\cP{{\cal P}}
\def\cF{{\cal F}}

\def\cG{{\cal G}}

\def\cV{{\cal V}}
\def\cW{{\cal W}}

\dimendef\dimen=0

\newtheorem{theorem}{Theorem}[section]
\newtheorem{lemma}[theorem]{Lemma}

\newtheorem{proposition}[theorem]{Proposition}
\newtheorem{remark}[theorem]{Remark}

\begin{document}

\noindent

~

\bigskip
\begin{center}
{\bf ON THE DOMINATION OF RANDOM WALK ON A DISCRETE CYLINDER BY RANDOM INTERLACEMENTS}
\end{center}

\begin{center}
Alain-Sol Sznitman
\end{center}


\bigskip
\begin{abstract}
We consider simple random walk on a discrete cylinder with base a large $d$-dimensional torus of side-length $N$, when $d \ge 2$. We develop a stochastic domination control on the local picture left by the random walk in boxes of side-length of order $N^{1- \ve}$, with $0 < \ve < 1$, at certain random times comparable to $N^{2d}$, in terms of the trace left in a similar box of $\IZ^{d+1}$ by random interlacements at a suitably adjusted level. As an application we derive a lower bound on the disconnection time $T_N$ of the discrete cylinder, which as a by-product shows the tightness of the laws of $N^{2d}/T_N$, for all $d \ge 2$. This fact had previously only been established when $d \ge 17$, in \cite{DembSzni08}.
\end{abstract}

\vfill 

\n
Departement Mathematik  \\
ETH Z\"urich\\
CH-8092 Z\"urich\\
Switzerland

\vfill
~
\newpage
\thispagestyle{empty}
~

\newpage
\setcounter{page}{1}

 \setcounter{section}{-1}
 
 \section{Introduction}
 \setcounter{equation}{0}
 
 The present article relates random walk on a discrete cylinder with base a $d$-dimensional torus, $d \ge 2$, of large side-length $N$ to the model of random interlacements recently introduced in \cite{Szni07a}. It develops a stochastic domination control on the trace left by the random walk in boxes of side-length of order $N^{1- \ve}$ in the cylinder at times which are comparable to $N^{2d}$, in terms of the trace left by random interlacements at a suitably adjusted level in a box of $\IZ^{d+1}$ with same side-length. As an application of this stochastic domination control and of estimates from \cite{SidoSzni09a} on the percolative character of the vacant set left by random interlacements at a small level $u$, we derive a lower bound on the disconnection time $T_N$ of the discrete cylinder by simple random walk. In particular our bounds imply that the laws of the variables $N^{2d}/T_N$ are tight, for all $d \ge 2$. This result was previously only known to hold when $d \ge 17$, cf.~\cite{DembSzni08}. Combined with the upper bounds of \cite{Szni08b}, this shows that for all $d \ge 2$, ``$T_N$ lives in scale $N^{2d}$''.
 
 \medskip
 We will now present the objects of study more precisely. For $d \ge 2$ and $N \ge 1$, we consider the discrete cylinder
 \begin{equation}\label{0.1}
 E = \IT \times \IZ, \;\;\mbox{where} \;\; \IT = (\IZ / N \IZ)^d\,.
 \end{equation}

\n
For $x$ in $E$ we denote with $P_x$, resp.~$P$, the canonical law on the space $\cT$ of nearest-neighbor $E$-valued trajectories, of the simple random walk on $E$ starting at $x$, resp.~with the uniform distribution on $\IT \times \{0\}$. We write $X_\point$ for the canonical process and $Y_\point$ and $Z_\point$ for its respective $\IT$ and $\IZ$ components.

\medskip
Another important ingredient are the so-called random interlacements at level $u \ge 0$ introduced in \cite{Szni07a}. They describe the trace on $\IZ^{d+1}$ (where $d + 1$ in the present article plays the role of $d$ in \cite{Szni07a}) left by a cloud of paths constituting a Poisson point process on the space of doubly infinite trajectories on $\IZ^{d+1}$ modulo time-shift, tending to infinity at positive and negative infinite times. We refer to Section 1 for precise definitions. The non-negative parameter $u$ essentially corresponds to a multiplicative factor of the intensity measure of this point process. In a standard fashion one constructs on the same space $(\Omega, \cA, \IP)$, see (\ref{1.14}), (\ref{1.20}), the family $\cI^u, u \ge 0$, of random interlacements at level $u$. They are the traces on $\IZ^{d+1}$ of the trajectories modulo time-shift in the cloud, which have labels at most $u$. The random subsets $\cI^u$ increase with $u$, and for $u > 0$ constitute infinite random connected subsets of $\IZ^{d+1}$, ergodic under space translations, cf.~Theorem 2.1 and Corollary 2.3 of \cite{Szni07a}. The complement $\cV^u$ of $\cI^u$ in $\IZ^{d+1}$ is the so-called vacant set at level $u$.

\medskip
Our main result establishes a stochastic domination control on scales of order $N^{1-\ve}$, $0 < \ve < 1$, of the local picture left by simple random walk on the cylinder $E$ at certain random times, in terms of the corresponding trace of a random interlacement $\cI^v$, at a suitably adjusted level $v$. More precisely, given a height $z \in \IZ$ in the cylinder, we consider the sequence $R^z_k, D^z_k, k \ge 1$, of successive return times of the vertical component of the walk to an interval of length of order $N$ centered at $z$ and departures from a concentric interval of length of order $N(\log N)^2$, cf.~(\ref{1.10}). We show in the main Theorem \ref{theo1.1}, that for $0 < \ve < 1$, $\alpha > 0$, $v > (d+1) \,\alpha$, for large $N$, given any $x = (y,z)$ in $E$, we can construct a probability $Q$ on some auxiliary space coupling the simple random walk on $E$ under $P$, with the random interlacements on $\IZ^{d+1}$ under $\IP$, so that, cf.~(\ref{1.24}),
 \begin{equation}\label{0.2}
Q [(X_{[0,D^z_K]} - x) \cap A \subseteq \cI^v \cap A] \ge 1 - c\,N^{-3d}\,,
 \end{equation}
 
 \n
where $K$ has order $\alpha N^{d-1} (\log N)^{-2}$, $A$ is a box centered at the origin with side-length of order $N^{1-\ve}$, (viewed both as subset of $E$ and $\IZ^{d+1}$), and $c$ a dimension dependent constant.

\medskip
When $z$ has size of order at most $N^d$, the random times $D^z_K$, which appear in (\ref{0.2}) have typical order of magnitude $N^{2d}$, cf.~Proposition \ref{prop7.1} and Remark \ref{rem7.2}. As an application of the main Theorem \ref{theo1.1} we derive a lower bound on the disconnection time $T_N$ of the discrete cylinder by simple random walk, cf.~(\ref{7.1}). Namely we show in Theorem \ref{theo7.3} that
\begin{equation}\label{0.3}
\underset{N}{\underline{\lim}}\; P[T_N > \gamma N^{2d}] \ge W\Big[\zeta \Big(\mbox{\f $\dis\frac{v}{\sqrt{d+1}}$}\Big) > \gamma\Big], \;\mbox{for all $\gamma > 0$}\,,
\end{equation}
where $v$ is a suitably small number, $W$ stands for the Wiener measure, and
\begin{equation}\label{0.4}
\zeta(u) = \inf\big\{t \ge 0; \;\sup\limits_{a \in \IR} \;L(a,t) \ge u\big\}, \;\mbox{for $u \ge 0$}\,,
\end{equation}

\n
with $L(a,t)$ a jointly continuous version of the local time of the canonical Brownian motion. In particular this implies that for $d \ge 2$, 
\begin{equation}\label{0.5}
\mbox{the laws of $N^{2d}/T_N$ under $P$, $N \ge 2$, are tight},
\end{equation}

\n
a property previously only established when $d \ge 17$, cf.~\cite{DembSzni08}.

\medskip
It is an open problem, cf.~Remark 4.7 2) of \cite{Szni08b}, whether in fact
\begin{equation}\label{0.6}
\mbox{$T_N/N^{2d}$ converges in law towards $\zeta \Big(\mbox{\f $\dis\frac{u_*}{\sqrt{d+1}}\Big)$}$},
\end{equation}

\n
where $u_* \in (0,\infty)$ is the critical value for the percolation of $\cV^u$, cf.~\cite{Szni07a}, \cite{SidoSzni09a}, see also in the present work below (\ref{1.22}). A companion upper bound to (\ref{0.3}) already appears in Corollary 4.6 of \cite{Szni08b} and states that
\begin{equation}\label{0.7}
\overline{\lim\limits_N} \;P[T_N \ge \gamma N^{2d}] \le W \Big[ \zeta \Big(\mbox{\f $\dis\frac{u_{**}}{\sqrt{d+1}}$}\Big) \ge \gamma\Big], \;\mbox{for all} \;\gamma > 0,
\end{equation}

\n
where $u_{**} \in [u_*,\infty)$ is another critical value, cf.~(0.6) of \cite{Szni08b}, and Remark 7.5 2) below. The claim (\ref{0.6}) would follow from proving (\ref{0.3}) with $v = u_*$, and showing that $u_* = u_{**}$. In this respect an additional interest of Theorem \ref{theo1.1} stems from the fact that it enables to improve the value $v$ in (\ref{0.3}) once quantitative controls on the percolative properties of the vacant set $\cV^u$, with $u < u_*$, are derived. We refer to Remark 7.5 2) for further discussion of this matter. As a direct consequence of (\ref{0.5}) and of the upper bound (0.7) of \cite{Szni08b}, one thus finds that for all $d \ge 2$,
\begin{equation}\label{0.8}
\begin{array}{l}
\mbox{the laws on $(0,\infty)$ of $T_N/N^{2d}$ under $P$, with $N \ge 2$, are tight,}
\\
\mbox{i.e. ``$T_N$ lives in scale $N^{2d}$''}.
\end{array}
\end{equation}

\medskip\n
We will now give some comments on the proofs of the main results. The derivation of Theorem \ref{theo1.1} involves a sequence of steps which combine some of the techniques which have been developed in \cite{Szni07a}, \cite{Szni09a} and \cite{Szni08b}. A more detailed outline of these steps appears in Section 1 after the statement of Theorem \ref{theo1.1}. For the time being we only discuss the rough strategy of the proof, and for simplicity assume that $x = 0$ in (\ref{0.2}). We also write $R_k,D_k$ in place of $R_k^{z=0}, D_k^{z=0}$, for $k \ge 1$, see above (\ref{0.2}). A key identity proved in Lemma 1.1 of \cite{Szni08b} and recalled in (\ref{1.13}) below, makes it advantageous to replace the true excursions $X_{(R_k + \cdot) \wedge D_k}$, $k \ge 1$, which contain all the information about $X_{[0,D_K]} \cap A$, with an iid collection $X_\point^{\prime k}$, $1 \le k \le K^\prime$, of ``special excursions'', which have same distribution as the walk starting uniformly on the collection of points in $E$ with height equal to $\pm \;N$, stopped when exiting $\wt{B} = \IT \times (-h_N,h_N)$, where $h_N$ is of order $N(\log N)^2$, cf.~(\ref{1.8}), and $K^\prime$ is slightly bigger than $K$, cf.~(\ref{1.26}). Indeed one then uses a Poissonization procedure and only retains excursions entering $A$, from the time they enter $A$ until they exist $\wt{B}$. In this fashion one obtains a Poisson point measure $\mu^\prime$ on the set of paths starting on the ``surface of $A$'', and stopped at the boundary of $\wt{B}$, cf.~(\ref{4.1}). The intensity measure of this Poisson point measure has a structure similar to the intensity measure of the Poisson point process attached to trajectories of a random interlacement entering $A$, cf.~Proposition \ref{prop4.1} and (\ref{1.18}), (\ref{1.19}). This eventually leads to the desired comparison.

\medskip
To replace the true excursions with the ``special excursions'', we proceed as follows. The coupling technique of Proposition \ref{prop2.2}, see also Section 3 of \cite{Szni08b}, takes advantage of the fact that between each departure of $\wt{B}$ and return to the much smaller $B = \IT \times [-N,N] \subseteq \wt{B}$, the $\IT$-component of the walk has sufficient time to homogenize. This enables us to replace the true excursions $X_{(R_k + \cdot) \wedge D_k}$, $1 \le \cdot \le K$, with a collection of excursions $\wt{X}_\point^k$, $1 \le k \le K$, which however are not iid. These excursions are only independent conditionally on the sequences $Z_{R_k},Z_{D_k}, k \ge 1$, with respective laws given by that of a ``special excursion'', (see above), conditioned to start at height $Z_{R_k}$ and exit $\wt{B}$ at height $Z_{D_k}$. One then needs to dominate the ranges of $\wt{X}_\point^k$, $1 \le k \le K$ with the ranges of a collection of iid ``special excursions'' $X_\point^{\prime k}$, $1 \le k \le K^\prime$, (with $K^\prime$ slightly bigger than $K$, as mentioned above). This step is achieved by constructing a suitable coupling in Section 3, and using large deviation estimates under $P$ for the pair empirical distribution $\frac{1}{K} \;\sum_{k \le K} \;\delta_{(Z_{R_k},Z_{D_k})}$, and for a similar object attached to the iid ``special excursions'', with $K^\prime$ in place of $K$. The above pair empirical distribution attached to $Z_{R_k}$, $Z_{D_k}$, $k \ge 1$, under $P$, can be controlled with the pair empirical distribution recording consecutive values of a Markov chain on $\{1, -1\}$, with $N$-dependent transition probability, governing the evolution of ${\rm sign}(Z_{D_k}) = {\rm sign}(Z_{R_{k+1}})$, $P$-a.s.. The crucial domination estimate appears in Proposition \ref{prop3.1}.

\medskip
As already pointed out, once true excursions are replaced with ``special excursions'', one is quickly reduced to the consideration of the trace on $A$ of the paths in the support of a Poisson point measure $\mu^\prime$ with state space the set of excursions starting on the surface of $A$ and stopped at the boundary of $\wt{B}$. However these excursions live on a slice of the cylinder $E$ and not on $\IZ^{d+1}$. To correct this feature and enable a comparison with random interlacements, we employ truncation as well as the ``sprinkling technique'' of \cite{Szni07a}. Namely we only retain the part of the excursions going from their starting point on the surface of $A$ up to their first exit from a box $\wt{C}$ of side-length of order $\frac{N}{2}$ centered at the origin, cf.~(\ref{1.27}). This is the truncation. We also slightly increase the intensity of the Poisson measure. This slight increase of the intensity, the ``sprinkling'', is meant to compensate for the truncation of the original excursions as they exit $\wt{C}$, and ensure that the trace on $A$ of the trajectories in the support of this new Poisson point measure $\mu$ typically dominates the corresponding trace on $A$ of paths in the support of $\mu^\prime$. The key control appears in Proposition \ref{prop5.1}. This result is similar up to some modifications to Theorem 3.1 of \cite{Szni08b}, where truncation and sprinkling is carried out on $\IZ^{d+1}$-valued trajectories instead of $E$-valued trajectories here. The interest of the step we just described is that paths in the support of the Poisson point measure $\mu$ live in $\wt{C} \cup \partial \wt{C}$, which can both be viewed as a subset of $E$ and $\IZ^{d+1}$. The intensity measure of $\mu$, cf.~(\ref{5.4}), can easily be compared to the intensity of the Poisson point measure $\mu_{A,v}$, cf.~(\ref{1.18}), which contains the information of the trace on $A$ left by random interlacements at level $v$. This is the essence of the comparison which appears in Proposition \ref{prop6.1} and leads to the conclusion of the proof of Theorem \ref{theo1.1}.

\medskip
The lower bound on the disconnection time $T_N$, cf.~(\ref{0.3}) or Theorem \ref{theo7.3}, now follows rather straightforwardly. It relies on the one hand on estimates for the random times $D^z_K$ which relate them to the random variable $\zeta$ of (\ref{0.4}), see Proposition \ref{prop7.1} and Remark \ref{rem7.2}, and on the other hand on the fact, see (\ref{7.15}), that
\begin{equation}\label{0.9}
\lim\limits_N \;P[T_N \le \gamma N^{2d} < \inf\limits_{|z| \le N^{2d+1}} D^z_K] = 0\,,
\end{equation}

\n
when the parameter $\alpha$ entering the definition of $K$, cf.~(\ref{1.24}) is chosen small enough. To prove (\ref{0.9}) one uses Theorem \ref{theo1.1} as well as controls from \cite{SidoSzni09a} on the rarity of long planar $*$-paths in $\cI^v$, when $v$ is small, see (\ref{1.23}) below. The point is that the occurrence of the disconnection before time $\gamma N^{2d}$ forces the presence somewhere in the cylinder at height in absolute value at most $N^{2d+1}$, of a long planar $*$-path in $X_{[0,T_N]}$, cf.~Lemma \ref{lem7.4}. Let us mention that being able to prove (\ref{0.9}) for all $\alpha < \frac{u_*}{d+1}$ would yield (\ref{0.3}) with $v = u_*$, and thus bring one closer to a proof of (\ref{0.6}), see also Remark 7.5 2).

\medskip
We will now describe the organization of this article.

\medskip
Section 1 introduces further notation and recalls various useful facts concerning random walks and random interlacements. The main Theorem \ref{theo1.1} is stated and an outline of the main steps of its proof is provided.

\medskip
In Section 2 we construct the excursions $\wt{X}_\point^k$, $k \ge 1$, mentioned in the above discussion. The main result appears in Proposition \ref{prop2.2}.

\medskip
Section 3 shows how one can dominate the ranges $X_{[R_k,D_k]}$, $1 \le k \le K$, in terms of the ranges of an iid collection of ``special excursions'' $X_\point^{\prime k}$, $1 \le k \le K^\prime$, where $K^\prime$ is slightly bigger than $K$. The key control appears in Proposition \ref{prop3.1}. 

\medskip
Section 4 contains a Poissonization step where the Poisson point measure $\mu^\prime$ is introduced.

\medskip
In Section 5 truncation and sprinkling enable to dominate the trace on $A$ of the paths in the support of $\mu^\prime$ in terms of the corresponding trace of the truncated paths in the support of the Poisson point measure $\mu$. The main step is Proposition \ref{prop5.1}. The Proposition \ref{prop5.4} comes as a direct consequence and encapsulates what is needed for the next section. 

\medskip
Section 6 develops the final comparison between random walk on $E$ and random interlacements on $\IZ^{d+1}$, so as to complete the proof of Theorem \ref{theo1.1}.

\medskip
In Section 7 we give an application to the derivation of a lower bound on the disconnection time in Theorem \ref{theo7.3}. Some open problems are mentioned in Remark \ref{rem7.5}.

\medskip
Let us comment on the convention we use for constants. Throughout the text $c$ or $c^\prime$ denote positive constants solely depending on $d$, with values changing from place to place. The numbered constants $c_0,c_1,\dots$ are fixed and refer to the value at their first appearance in the text. Dependence of constants on additional parameters appear in the notation. For instance $c(\alpha)$ stands for a positive constant depending on $d$ and $\alpha$.

Finally some pointers to the literature on random interlacements might be useful to the reader. Random interlacements on $\IZ^{d+1}$ have been introduced in \cite{Szni07a}, where the investigation of the percolative properties of the vacant set was initiated. The uniqueness of the infinite cluster of the vacant set has been shown in \cite{Teix09a}, and the positivity of $u_*$ in full generality in \cite{SidoSzni09a}, (in \cite{Szni07a} this had only been shown when $d  \ge 6$). The stretched exponential decay of the connectivity function for $u > u_{**}$, is proved in \cite{SidoSzni09b}, and quantitative controls on the rarity of large finite clusters in the vacant set, when $d  \ge 4$ and $u$ sufficiently small, are developed in \cite{Teix08c}. Random interlacements on transient weighted graphs are discussed in \cite{Teix08b}. The fact that random interlacements describe the microscopic structure left by random walks on discrete cylinders at times comparable to the square of the number of points of the base is the object of \cite{Szni09a}. Similar results for the random walk on the torus, and generalizations to cylinders with more general bases can respectively be found in \cite{Wind08}, and \cite{Wind09}. Applications of random interlacements to the control of the disconnection time of discrete cylinders are the main theme of \cite{Szni08b}, where an upper bound on the disconnection time is derived, and of the present article, where a lower bound on the disconnection time is obtained.

\section{Some notation and the main result}
\setcounter{equation}{0}

In this section we introduce additional notation and recall some useful results concerning random walks and random interlacements. In particular a key identity from Lemma 1.1 of \cite{Szni08b} for the hitting distribution of the walk on the cylinder lies at the heart of the comparison with random interlacements. We recall it below in (\ref{1.13}). We then state the main Theorem \ref{theo1.1} and outline the key steps of its proof.

\medskip
We write $\IN = \{0,1,2,\dots \}$ for the set of natural numbers. Given a non-negative real number $a$, we write $[a]$ for the integer part of $a$, and for real numbers $b,c$ we write $b \wedge c$ and $b \vee c$ for the respective minimum and maximum of $b$ and $c$. We write $e_i, 1 \le i \le d+1$, for the canonical basis of $\IR^{d+1}$. We let $|\cdot |$ and $|\cdot |_\infty$ respectively stand for the Euclidean and $\ell^\infty$-distances on $\IZ^{d+1}$ or for the corresponding distances induced on $E$. Throughout the article we assume $d \ge 2$. We say that two points on $\IZ^{d+1}$ or $E$ are neighbors, respectively $*$-neighbors, if their $|\cdot |$-distance, respectively $| \cdot |_\infty$-distance equals $1$. By finite path, respectively finite $*$-path, we mean a finite sequence $x_0,x_1,\dots,x_n$ on $\IZ^{d+1}$ or $E$, $n \ge 0$, such that for each $0 \le i < n$, $x_i$ and $x_{i+1}$ are neighbors, respectively $*$-neighbors. Sometimes, when this causes no confusion, we simply write path or $*$-path, in place of finite path or finite $*$-path. We denote the closed $|\cdot|_\infty$-ball and the $|\cdot |_\infty$-sphere with radius $r \ge 0$ and center $x$ in $\IZ^{d+1}$ or $E$ with $B(x,r)$ and $S(x,r)$. For $A,B$ subsets of $\IZ^{d+1}$ or $E$ we write $A + B$ for the set of elements $x + y$ with $x$ in $A$ and $y$ in $B$. We also write $U \subset\subset \IZ^{d+1}$ or $U \subset \subset E$ to indicate that $U$ is a finite subset of $\IZ^{d+1}$ or $E$. Given $U$ subset of $\IZ^{d+1}$ or $E$, we denote with $|U|$ the cardinality of $U$, with $\partial U$ the boundary of $U$ and $\partial_{\rm int} U$ the interior boundary of $U$:
\begin{equation}\label{1.1}
\partial U = \{ x \in U^c; \,\exists x^\prime \in U, |x- x^\prime | = 1\}, \;\;\partial_{\rm int} U = \{x \in U; \,\exists x^\prime \in U^c, \,|x - x^\prime| = 1\}\,.
\end{equation}

\n
We write $\pi_\IT$ and $\pi_\IZ$ for the respective canonical projections from $E = \IT \times \IZ$ onto $\IT$ \linebreak and $\IZ$.

\medskip
We let $\cT$ stand for the set of nearest neighbor $E$-valued trajectories with time indexed by $\IN$, see below (\ref{0.1}). When $F$ is a subset of $E$, or of $\IZ^{d+1}$, we denote with $\cT_F$ the countable set of nearest neighbor $(F \cup \partial F)$-valued trajectories which remain constant after a finite time. The canonical shift on $\cT$ is denoted with $(\theta_n)_{n \ge 0}$ and the canonical filtration with $(\cF_n)_{n \ge 0}$. Further notation concerning the canonical process on $\cT$ appears below (\ref{0.1}). Given a subset $U$ of $E$ we denote with $H_U$, $\wt{H}_U$ and $T_U$, the respective entrance time of $U$, hitting time of $U$, and exit time from $U$:
\begin{equation}\label{1.2}
\begin{split}
H_U & = \inf\{n \ge 0; X_n \in U\}, \;\wt{H}_U = \inf\{n \ge 1, X_n \in U\}\,,
\\[1ex]
T_U & = \inf\{n \ge 0; \,X_n \notin U\}\,.
\end{split}
\end{equation}

\n
In the case of a singleton $U = \{x\}$, we simply write $H_x$ or $\wt{H}_x$.

\medskip
We denote with $P_x^{\IZ^{d+1}}$ the canonical law of simple random walk on $\IZ^{d+1}$ starting at $x$ and with $E_x^{\IZ^{d+1}}$ the corresponding expectation. We otherwise keep the same notation as for the walk on $E$ concerning the canonical process, the canonical shift and natural objects such as in (\ref{1.2}). Given $K \subset\subset \IZ^{d+1}$ and $U \supseteq K$, a subset of $\IZ^{d+1}$, the equilibrium measure and the capacity of $K$ relative to $U$ are defined by:
\begin{align}
e_{K,U}(x) & =P_x^{\IZ^{d+1}} \,[\wt{H}_K > T_U], \;\;\mbox{for $x \in K$}\,, \label{1.3}
\\[1ex]
& = 0, \;\;\mbox{for $x \notin K$, and} \nonumber
\\[1ex]
{\rm cap}_U(K) & = \dsl_{x  \in K} \,e_{K,U}(x) (\le |K|)\,. \label{1.4}
\end{align}

\n
The Green function of the walk killed outside $U$ is defined as
\begin{equation}\label{1.5}  
g_U(x,x^\prime) = E_x^{\IZ^{d+1}} \Big[\dsl_{n \ge 0} 1\{X_n = x^\prime, n < T_U\}\Big], \;\mbox{for $x,x^\prime$ in $\IZ^{d+1}$}\,.
\end{equation}

\medskip\n
When $U = \IZ^{d+1}$, we drop $U$ from the notation in (\ref{1.3}) - (\ref{1.5}). The Green function is symmetric in its two variables and the probability to enter $K$ before exiting $U$ can be expressed as:
\begin{equation}\label{1.6}
P_x^{\IZ^{d+1}} [H_K < T_U] = \dsl_{x^\prime \in \IZ^{d+1}} g_U (x,x^\prime) \,e_{K,U}(x^\prime), \;\mbox{for $x \in \IZ^{d+1}$}\,.
\end{equation}

\n
One also has the bounds, (see for instance (\ref{1.7}) of \cite{Szni08b}):
\begin{equation}\label{1.7}
\begin{array}{l}
\dsl_{x^\prime \in K} \,g_U (x,x^\prime) / \sup\limits_{y \in K } \;\dsl_{x^\prime \in K} \;g_U (y,x^\prime)   \le 
 P_x^{\IZ^{d+1}} [H_K < T_U] \le 
 \\[4ex]
 \dsl_{x^\prime \in K} \;g_U (x,x^\prime) / \inf\limits_{y \in K} \;\dsl_{x^\prime \in K} g_U(y,x)\,.
\end{array}
\end{equation}

\medskip\n
In the case of the discrete cylinder $E$, when $U \subsetneq E$ is a strict subset of $E$, we define the corresponding objects just as in (\ref{1.3}) - (\ref{1.5}), with $P_x$ and $E_x$ in place of $P_x^{\IZ^{d+1}}$ and $E_x^{\IZ^{d+1}}$. We then have similar identities and bounds as in (\ref{1.6}), (\ref{1.7}). When $\rho$ is a measure on $E$ or $\IZ^{d+1}$, we write $P_\rho$ or $P_\rho^{\IZ^{d+1}}$ in place of $\sum_{x \in E} \rho(x) \,P_x$ or $\sum_{x \in \IZ^{d+1}} \,\rho(x) \,P_x^{\IZ^{d+1}}$.

\medskip
As mentioned above (\ref{0.2}) the main Theorem \ref{theo1.1} involves measuring time in terms of excursions of the random walk in and out of certain concentric boxes in the cylinder $E$. More specifically we introduce the vertical scales
\begin{equation}\label{1.8}
r_N = N < h_N = [N(2 + (\log N)^2)]\,,
\end{equation}

\n
and the boxes in $E$ centered at level $z \in \IZ$:
\begin{equation}\label{1.9}
\begin{array}{l}
B(z) = \IT \times (z + I) \subseteq \wt{B}(z) = \IT \times (z + \wt{I}), \;\;\mbox{where}
\\
I = [-r_N,r_N] \;\;\mbox{and} \;\; \wt{I} = (-h_N,h_N)\,.
\end{array}
\end{equation}

\n
When $z = 0$, we simply write $B$ and $\wt{B}$. The sequence of successive returns of $X_\point$ to $B(z)$ and departure from $\wt{B}(z)$, $R^z_k$, $D^z_k$, $k \ge 1$, is then defined via:
\begin{equation}\label{1.10}
\begin{array}{l}
R^z_1 = H_{B(z)}, \;D^z_1 = T_{\wt{B}(z)} \circ \theta_{R^z_1} + R^z_1, \;\mbox{and for $k \ge 1$},
\\
R^z_{k+1} = R^z_1 \circ \theta_{D^z_k} + D^z_k, \;\mbox{and} \;D^z_{k+1} = D_1^z \circ \theta_{D^z_k} + D^z_k\,,
\end{array}
\end{equation}

\n
so that $0 \le R^z_1 \le D_1^z \le \dots \le R^z_k \le D^z_k \le \dots \le \infty$, and these inequalities except maybe for the first one are $P$-a.s. strict. When $z = 0$, we simply write $R_k,D_k$ in place of $R^0_k$, $D^0_k$, for $k \ge 1$.

\medskip
Certain initial distributions of the walk on $E$ will be useful in what follows. Namely, we will consider for $z \in \IZ$:
\begin{equation}\label{1.11}
q_z = \mbox{\f $\dis\frac{1}{N^d}$}  \dsl_{x \in \IT \times \{z\}} \delta_x, \;\;\mbox{as well as} \;\;q = \fr \;(q_{r_N} + q_{-r_N})\,.
\end{equation}

\n
As a result of Lemma 1.1 of \cite{Szni08b}, the initial distribution $q$ plays a central role in linking random walk on $E$ and random interlacements, see also Remark 1.2 of \cite{Szni08b}. Indeed for $K \subseteq \IT \times (-r_N,r_N)$, one has:
\begin{align}
&P_q[H_K < T_{\wt{B}}, \,X_{H_K} = x] = (d+1) \;\dis\frac{(h_N - r_N)}{N^d} \;e_{K,\wt{B}}(x), \;\mbox{for $x \in K$}, \label{1.12}
\intertext{\mbox{and with the application of the strong Markov property,}}
&P_q[H_K < T_{\wt{B}}, (X_{H_K + \,\cdot}) \in dw] = (d+1) \;\dis\frac{(h_N - r_N)}{N^d} \,P_{e_{K,\wt{B}}} (dw) \,.\label{1.13}
\end{align}

\n
We will now recall some notation and results from \cite{Szni07a} concerning random interlacements. We denote with $W$ the space of doubly infinite nearest neighbor $\IZ^{d+1}$-valued trajectories which tend to infinity at positive and negative infinite times, and with $W^*$ the space of equivalence classes of trajectories in $W$ modulo time-shift. The canonical projection from $W$ onto $W^*$ is denoted by $\pi^*$. We endow $W$ with its canonical $\sigma$-algebra $\cW$, and denote by $X_n$, $n \in \IZ$, the canonical coordinates. 

\medskip
We endow $W^*$ with $\cW^* = \{A \subseteq W^*; (\pi^*)^{-1} (A) \in \cW\}$, the largest $\sigma$-algebra on $W^*$ for which $\pi^*: (W, \cW) \rightarrow (W^*, \cW^*)$ is measurable. We also consider $W_+$ the space of nearest neighbor $\IZ^{d+1}$-valued trajectories defined for non-negative times and tending to infinity. We write $\cW_+$ and $X_n$, $n \ge 0$, for the canonical $\sigma$-algebra and the canonical process on $W_+$. Since $d \ge 2$, the simple random walk on $\IZ^{d+1}$ is transient and $W_+$ has full measure for any $P_x^{\IZ^{d+1}}$, $x \in \IZ^{d+1}$, see above (\ref{1.3}), and we view whenever convenient the law of simple random walk on $\IZ^{d+1}$ starting from $x$, as a probability on $(W_+, \cW_+)$. We consider the space of point measures on $W^* \times \IR_+$:
\begin{equation}\label{1.14}
\Omega =  \left\{ \begin{array}{l} 
\omega = \dsl_{i \ge 0} \, \delta_{(w^*_i,u_i)}, \; \mbox{with $(w_i^*,u_i) \in W^* \times \IR_+, i \ge 0$, and}  
\\[0.5ex]
\mbox{$\o (W^*_K \times [0,u] ) < \infty$, for any $K \subset \subset \IZ^{d+1}, u \ge 0$}, 
\end{array}\right\}
\end{equation}

\n
where for $K \subset \subset \IZ^{d+1}$, $W^*_K \subseteq W^*$ is the subset of trajectories modulo time-shift, which enter $K$:
\begin{equation}\label{1.15}
\mbox{$W^*_K = \pi^*(W_K)$ and $W_K = \{w \in W$; for some $n \in \IZ$, $X_n(\omega) \in K\}$}.
\end{equation}

\n
We endow $\Omega$ with the $\sigma$-algebra $\cA$ generated by the evaluation maps $\o \r \o(D)$, where $D$ runs over the product $\sigma$-algebra $\cW^* \times \cB(\IR_+)$. We denote with $\IP$ the probability on $(\Omega, \cA)$ under which $\o$ becomes a Poisson point measure on $W^* \times \IR_+$ with intensity $\nu(dw^*)du$, giving finite mass to the sets $W^*_K \times [0,u]$, for $K \subset \subset \IZ^{d+1}$, $u \ge 0$. Here $\nu$ stands for the unique $\sigma$-finite measure on $(W^*,\cW^*)$ such that for every $K \subset \subset \IZ^{d+1}$, cf.~Theorem 1.1 of \cite{Szni07a}:
\begin{equation}\label{1.16}
1_{W^*_K} \,\nu = \pi^* \circ Q_K\,,
\end{equation}

\medskip\n
with $Q_K$ the finite measure on $W^0_K$, the subset of $W_K$ of trajectories which enter $K$ for the first time at time $0$, such that for $A,B$ in $\cW_+$, $x \in \IZ^{d+1}$:
\begin{equation}\label{1.17}
Q_K [(X_{-n})_{n \ge 0} \in A, \;X_0 = x, \;(X_n)_{n \ge 0} \in B] = P_x^{\IZ^{d+1}} [A \,|\,\wt{H}_K = \infty] \,e_K(x) \,P_x^{\IZ^{d+1}} [B]\,,
\end{equation}

\n
where $e_K$, cf.~(\ref{1.3}) and below (\ref{1.5}), stands for the equilibrium measure of $K$, and is concentrated on the points of $\partial_{\rm int} K$ for which $P_x [\wt{H}_K = \infty] > 0$.

\medskip
Given $K \subset \subset \IZ^{d+1}$, $u \ge 0$, one further defines on $(\Omega, \cA)$ the random point process with state space the set of finite point measures on $(W_+, \cW_+)$:
\begin{equation}\label{1.18}
\mu_{K,u}(\o) = \dsl_{i \ge 0} \,\delta_{(w_i^*)^{K,+}} \,1\{w^*_i \in W^*_K, u_i \le u\}, \;\mbox{for} \;\omega = \dsl_{i \ge 0} \,\delta_{(w_i^*,u_i)}\,,
\end{equation}

\n
where $(w^*)^{K,+}$ stands for the trajectory in $W_+$ which follows step by step $w^* \in W_K^*$ from the first time it enters $K$. One then has, cf.~Proposition 1.3 of \cite{Szni07a}, for $K \subset \subset \IZ^{d+1}$, $u \ge 0$:
\begin{equation}\label{1.19}
\begin{array}{l}
\mbox{$\mu_{K,u}$ is a Poisson point process on $(W_+, \cW_+)$ with intensity measure $u \,P_{e_K}^{\IZ^{d+1}}$}, 
\\
\mbox{where we used the notation introduced below (\ref{1.7})}\,.
\end{array}
\end{equation}

\medskip\n
Given $\o \in \Omega$, the interlacement at level $u \ge 0$, is the subset of $\IZ^{d+1}$:
\begin{equation}\label{1.20}
\begin{split}
\cI^u(\o) & = \bigcup\limits_{u_i \le u} \,{\rm range} (w_i^*), \;\;\mbox{if} \;\;\o = \dsl_{i \ge 0} \,\delta_{(w_i^*, u_i)}
\\[1ex]
= & \bigcup\limits_{K \subset \subset \IZ^{d+1}} \;\; \bigcup\limits_{w \in {\rm Supp} \,\mu_{K,u}(\o)} w(\IN)\,,
\end{split}
\end{equation}

\n
where for $w^* \in W^*$, range $(w^*) = w(\IZ)$ for any $w \in W$ with $\pi^*(w) = w^*$. The vacant set at level $u$ is then defined as:
\begin{equation}\label{1.21}
\cV^u(\omega) = \IZ^{d+1} \backslash \cI^u(\omega), \;\mbox{for} \;\omega \in \Omega, u \ge 0\,.
\end{equation}
One has the identity
\begin{equation}\label{1.22}
\IP[\cV^u \supseteq K] = {\rm exp}\{ - u \;{\rm cap}(K)\}, \;\mbox{for all $K \subset \subset \IZ^{d+1}$}\,,
\end{equation}

\medskip\n
and this property leads to a characterization of the law $Q_u$ on $\{0,1\}^{\IZ^{d+1}}$ of the random subset $\cV^u$, cf.~Remark 2.2 2) of \cite{Szni07a}.

\medskip
As a result of Theorem 3.5 of \cite{Szni07a} and 3.4 of \cite{SidoSzni09a}, it follows that there exists a non-degenerate critical value $u_* \in (0,\infty)$, such that for $u > u_*$, $\IP$-a.s., $\cV^u$ has only finite connected components, whereas for $u < u_*$, $\IP$-a.s., $\cV^u$ has an infinite connected component. It is also known, cf.~\cite{Teix09a}, that for each $u \ge 0$, there is $\IP$-a.s. at most one infinite connected component in $\cV^u$. The existence or absence of such a component when $u = u_*$ is presently an open problem. In Section 7, when applying Theorem 1.1 to the study of disconnection time we will also need the following estimate, cf.~(3.28) of \cite{SidoSzni09a}:
\begin{equation}\label{1.23}
\begin{array}{l}
\mbox{for any $\rho > 0$, there exists $u(\rho) > 0$ such that for $u \le u(\rho)$} 
\\
\mbox{$\lim\limits_{L \r \infty} \; L^\rho \,\IP[$ there is a $*$-path from $0$ to $S(0,L)$ in $\cI^u \cap (\IZ e_1 + \IZ e_{d+1})] = 0$}\,,
\end{array}
\end{equation}

\n
where we use the notation from the beginning of this section, and any $e_i, e_j, i \not= j$, could of course replace $e_1$ and $e_{d+1}$ in (\ref{1.23}).

\medskip
We can now state the main result of this article. It deals with the trace left in a neighborhood of size $N^{1-\varepsilon}$ of some point $x$ of the cylinder by the random walk at time $D^z_K$, where $z = \pi_\IZ(x)$ and $K$ has order $N^d/h_N$, cf.~(\ref{1.8}). Theorem \ref{theo1.1} shows that with high probability this trace is dominated by the corresponding trace of a random interlacement at a suitably adjusted level. When $|z|$ remains of order at most $N^d$, $D^z_K$ typically corresponds to time scales of order $N^{2d}$, cf.~Remark 7.2.

\begin{theorem}\label{theo1.1} $(d \ge 2$, $\alpha > 0$, $v > (d+1) \alpha$, $0 < \ve < 1)$

\medskip
For $N \ge c(\alpha, v, \ve)$ and $x = (y,z) \in E$ one can construct a coupling $Q$ on some auxiliary space of the simple random walk $X_\point$ on $E$ under $P$ and of the Poisson point measure $\o$ under $\IP$ so that
\begin{equation}\label{1.24}
Q[(X_{[0,D^z_K]} - x) \cap A \subseteq \cI^v(\o) \cap A] \ge 1 - cN^{-3d},
\end{equation}

\n
where $K = [\alpha N^d/h_N]$ and $A = B(0,N^{1-\ve})$ is viewed both as a subset of $E$ and $\IZ^{d+1}$.
\end{theorem}

The proof of Theorem \ref{theo1.1} involves several steps, which we now outline.

\begin{itemize}
\item[a)] This first step reduces the proof to the case where $x = 0$ and the initial distribution of the walk is $q_{z_0}$, cf.~(\ref{1.11}), with $z_0 \in I$, cf.~(\ref{1.9}).

\item[b)] This step constructs a coupling $Q_1$ of $X_\point$ under $P_{q_{z_0}}$ with a sequence $\wt{X}_\point^k$, $k \ge 1$, of $E$-valued processes, which are conditionally independent given $Z_{R_k}$, $Z_{D_k}$, $k \ge 1$, cf.~below (\ref{1.10}), with respective laws which coincide with that of $X_{\cdot \wedge T_{\wt{B}}}$ under $P_{Z_{R_k}, Z_{D_k}}$, where we use the notation:
\end{itemize}

\vspace{-5ex}
\begin{equation}\label{1.25}
P_{z_1,z_2} = P_{q_{z_1}} [ \cdot \,|\,Z_{T_{\wt{B}}} = z_2], \;\;\mbox{for $z_1 \in I, z_2 \in \partial \wt{I}$} \,,
\end{equation}

\begin{itemize}
\item[~~~]
and are such that, cf.~Proposition \ref{prop2.2},
\begin{equation*}
Q_1[\wt{X}_\point^k \not= X_{(R_k + \cdot) \wedge D_k} ] \le c\,N^{-4d}\,.
\end{equation*}

\item[c)] This steps constructs a coupling $Q_2$ of the above processes with a sequence of iid $E$-valued processes $X_\point^{\prime k}$, $k \ge 1$, with same law as $X_{\cdot \wedge T_{\wt{B}}}$ under $P_q$, cf.~(\ref{1.11}), in such a fashion that, cf.~Proposition \ref{prop3.1}, 
\begin{equation*}
Q_2\Big[\bigcup_{1 \le k \le K} X_{[R_k,D_k]} \subseteq \bigcup_{1 \le k \le K^\prime} \;{\rm range} \;X_\point^{\prime k}\Big] \ge 1 - c\,N^{-3d}\,,
\end{equation*}
where
\end{itemize}

\vspace{-2ex}
\begin{equation}\label{1.26}
K^\prime = \Big[\Big(1 + \mbox{\f $\dis\frac{2}{5}$}\;\delta\Big) \,\alpha \,N^d/h_N\Big] \;\; \mbox{and} \;\;1 + \delta = \Big(\mbox{\f $\dis\frac{v}{(d+1) \alpha}$} \Big) \wedge 2\,.
\end{equation}

\medskip
\begin{itemize}
\item[d)] This is a Poissonization step taking advantage of the special property of the distribution $q$, cf.~(\ref{1.12}), (\ref{1.13}). With $Q_3$ one couples the above processes with an independent Poisson variable $J^\prime$ of intensity $(1 + \frac{3}{5}\,\delta) \,\alpha \frac{N^d}{h_N}$, and defines the Poisson point measure on $\cT_{\wt{B}}$, cf.~below (\ref{1.1}),
\begin{equation*}
\mu^\prime = \dsl_{1 \le k \le J^\prime} \;\delta_{X_\point^{\prime k}} \,1 \{{\rm range} \;X_\point^{\prime k} \cap A \not= \emptyset\}\,,
\end{equation*}

\medskip\n
with intensity measure $(1 + \frac{3}{5}\,\delta) (d + 1) \,\alpha (1 - \frac{r_N}{h_N}) \,P_{e_{A, \wt{B}}} [X_{\cdot \wedge T_{\wt{B}}}  \in dw]$, as well as the random subset of $A$
\begin{equation*}
\cI^\prime = \bigcup\limits_{w \in {\rm Supp}\,\mu^\prime} \; {\rm range} \;w \cap A\,,
\end{equation*}

\medskip\n
where ${\rm Supp}\,\mu^\prime$ denotes the support (in $\cT_{\wt{B}})$ of the point measure $\mu^\prime$, so that, cf.~Proposition \ref{prop4.1}:
\begin{equation*}
Q_3 [X_{[0,D_K]} \cap A \subseteq \cI^\prime] \ge 1 - cN^{-3d}\,.
\end{equation*}

\item[e)] In this step one constructs using truncation and sprinkling a coupling $Q_4$ of $X_\point$ and $\cI^\prime$ under $Q_3$ with a Poisson point measure $\mu$ on $\cT_{\wt{C}}$, with intensity measure $(1 + \frac{4}{5}\,\delta) (d + 1) \,\alpha (1 -\frac{r_N}{h_N}) \,P_{e_{A,\wt{B}}} [X_{\cdot \wedge T_{\wt{C}}} \in dw]$, where
\end{itemize}

\vspace{-3ex}
\begin{equation}\label{1.27}
\wt{C} = B\Big(0, \mbox{\f $\dis\frac{N}{4}$}\Big)\,.
\end{equation}

\begin{itemize}
\item[~~~] Defining the random subset of $\wt{C} \cup \partial \wt{C}$
\begin{equation*}
\cI = \bigcup\limits_{w \in {\rm Supp} \,\mu} \;{\rm range} \,w\,,
\end{equation*}

\medskip\n
this coupling is such that, cf.~(\ref{5.44}) in the proof of Proposition \ref{prop5.4},
\begin{equation*}
Q_4 [\cI^\prime \subseteq \cI \cap A] \ge 1 - c\, N^{-3d}\,.
\end{equation*}

\item[f)] In this last step one constructs a coupling $Q^\prime$ of $X_\point$, $\cI^\prime, \cI$ under $Q_4$ with $\o$ under $\IP$ so that cf.~(\ref{6.5}),
\begin{equation*}
Q^\prime [X_{[0,D_K]} \cap A \subseteq \cI^v(\o) \cap A] \ge 1 - c N^{-3d}\,,
\end{equation*}

\n
and this enables to complete the proof of Theorem \ref{theo1.1}.
\end{itemize}

\begin{remark}\label{rem1.2} \rm
As it will be clear from the proof of Theorem \ref{theo1.1}, the exponent $-3d$ in the right-hand side of (\ref{1.24}) can be replaced by an arbitrary negative exponent by simply adjusting constants in Theorem \ref{theo1.1}. The specific choice of the exponent in (\ref{1.24}) will be sufficient for the application to the lower bound on the disconnection time we give in Section 7. \hfill $\square$
\end{remark}

\section{Reduction to the case $\pmb{x = 0}$ and a first coupling}
\setcounter{equation}{0}

This section takes care of steps a) an b)) in the above outline following the statement of Theorem \ref{theo1.1}. We first show in Proposition \ref{prop2.1} that it suffices to prove Theorem \ref{theo1.1} when $x = 0$ in (\ref{1.24}), and the initial distribution of the walk is $q_{z_0}$, with $z_0$ an arbitrary point on $I$, see (\ref{1.11}) and (\ref{1.9}). This is step a). Then we turn to step b) and construct, very much in the spirit of Proposition 3.3 of \cite{Szni09a}, a coupling of $X_\point$ with a sequence of $E$-valued processes $\wt{X}_\point^k$, $k \ge 1$, which conditionally on $Z_{R_k}, Z_{D_k}, k \ge 1$, are independent and respectively distributed as $P_{Z_{R_k}, Z_{D_k}}$, cf.~(\ref{1.25}), in such a fashion that each $\wt{X}^k_\point$ is close to $X_{(R_{k^+} \cdot) \wedge D_k}$. This construction is carried out in Proposition \ref{prop2.2}. It uses the fact that $h_N$ in (\ref{1.8}) is sufficiently large to provide ample time to the $\IT$-component of the walk to ``homogenize'' before reaching $B$, when the starting point of the walk lies outside $\wt{B}$.

\medskip
We keep the notation of Theorem \ref{theo1.1} and begin with the reduction to the case $x = 0$.

\begin{proposition}\label{prop2.1}
If for $N \ge c_0 (\ve, \alpha, v)$ and any $z_0 \in I$ one can construct a coupling $Q^\prime$ of $X_\point$ under $P_{q_{z_0}}$ with $\o$ under $\IP$ so that
\begin{equation}\label{2.1}
Q^\prime [X_{[0,D_K]} \cap A \subseteq \cI^v(\o)] \ge 1 - c N^{-3d}\,,
\end{equation}
then Theorem {\rm \ref{theo1.1}} follows.
\end{proposition}

\begin{proof}
Consider $N \ge c_0 (\ve, \alpha, v)$ and $x = (y,z)$ in $E$. Setting $\wh{X}_\point = X_{R_1^z + \cdot} - x$, and denoting with $\wh{R}_k, \wh{D}_k, k \ge 1$, the successive return times to $B$ and departure from $\wt{B}$ of $\wh{X}_\point$, one finds that
\begin{equation*}
(X_{[0,D^z_K]} - x) \cap A = \wh{X}_{[0,\wh{D}_K]} \cap A\,,
\end{equation*}
and moreover  that
\begin{equation*}
\begin{array}{l}
\mbox{$\wh{X}_\point$ is distributed as $P_{q_{z_0}}$, where $z_0$ coincides with $-z$, when $z \in I$, and}
\\
\mbox{otherwise with $r_N$ or $- r_N$.}
\end{array}
\end{equation*}

\n
With the coupling $Q^\prime$ mentioned in Proposition \ref{2.1} we can construct a conditional distribution under $Q^\prime$ of $\omega \in \Omega$ given $X_{[0,D_K]} \cap A$, which only takes finitely many values, and has same distribution under $Q^\prime$ as $\wh{X}_{[0,\wh{D}_K]} \cap A = (X_{[0,D_K^z]} - x) \cap A$, under $P$.  This conditional distribution and this identity in law enable to construct a coupling $Q$ of $X$ under $P$ with $\omega$ under $\IP$ so that (\ref{1.24}) holds as a result of (\ref{2.1}).
\end{proof}

We will now carry out step b) of the outline below Theorem \ref{theo1.1}. With Lemma 3.1 and Remark 3.2 of \cite{Szni08b}, we know that for $N \ge 1$,
\begin{equation}\label{2.2}
|P_x [X_{R_1} = x^\prime] - N^{-d} \,| \,\le c \,N^{-5d}, \;\mbox{for all $x \in \partial \wt{B}$, $x^\prime \in \partial_{\rm int} B$ with $\pi_\IZ (x) \,\pi_\IZ (x^\prime) > 0$}\,.
\end{equation}

\n
As mentioned in Remark 3.2 of \cite{Szni09a} the exponent $-5d$ in the right-hand side of (\ref{2.2}) can be replaced by an arbitrarily large negative exponent by adjusting constants. 

\medskip
The following proposition is simpler but has a similar spirit to Proposition 3.3 of \cite{Szni09a}. It will complete step b).

\begin{proposition}\label{prop2.2} $(N \ge 1, z_0 \in I)$

\medskip
One can construct on some auxiliary space $(\Omega_1,\cA_1,Q_1)$ a $\IZ$-valued process $Z$ and $\IT$-valued processes $Y_\point$, $\wt{Y}_\point^k$, $k \ge 2$, such that
\begin{align}
&\mbox{$X_\point = (Y_\point, Z_\point)$ under $Q_1$ has same law as $X_\point$ under $P_{q_{z_0}}$}\,, \label{2.3}
\\[1ex]
&\mbox{conditionally on $Z_{R_k}, Z_{D_k}, k \ge 1$, $\wt{X}^k_\point = X_{\cdot \wedge D_1}$, when $k = 1, = (\wt{Y}^k_\point, Z_{(R_k +\, \cdot) \wedge D_k})$} ,\label{2.4}
\\[-1ex]
&\mbox{when $k \ge 2$, are independent with same law as $X_{\cdot \wedge D_1}$ under $P_{Z_{R_k}, Z_{D_k}}$, cf.~{\rm (\ref{1.25})},} \nonumber
\\[2ex]
&Q_1 [\wt{X}_\point^k \not= X_{(R_k + \, \cdot) \wedge D_k}] \le c N^{-4d},\;\mbox{for} \; k \ge 1\,.\label{2.5}
\end{align}
\end{proposition}

\begin{proof}
It follows from (\ref{2.2}) that for $x \in \partial \wt{B}$ the total variation distance of the law of $X_{R_1}$ under $P_x$ and $q_{z(x)}$, where $|z(x)| = r_N$ and $\pi_\IZ(x) \cdot z(x) > 0$, is at most $N^d \,c\, N^{-5d} = c\,N^{-4d}$. With Theorem 5.2, p.~19 of \cite{Lind92}, we can construct for any $x \in \partial \wt{B}$ a probability $\rho_x(dx^\prime,d\wt{x})$ on $\{(x^\prime,\wt{x}) \in E^2$; $\pi_\IZ(x^\prime) = \pi_\IZ(\wt{x}) = z(x)\}$, such that under $\rho_x$
\begin{align}
&\mbox{the first marginal has same law as $X_{R_1}$ under $P_x$}\,,\label{2.6}
\\[2ex]
&\mbox{the second marginal is $q_{z(x)}$-distributed},\label{2.7}
\\[2ex]
&\rho_x (\{x^\prime \not= \wt{x}\}) \le c\,N^{-4d}\,. \label{2.8}
\end{align}

\medskip\n
We define the spaces $W_\IZ$, $W_\IT$ of respectively $\IZ$- and $\IT$-valued trajectories with jumps of $|\cdot |$-size at most $1$, as well as $W_\IZ^f$ and $W^f_\IT$ the countable subsets of $W_\IZ$ and $W_\IT$ of trajectories which are constant after a finite time. We pick the auxiliary space $\Omega_1 = W_\IT \times W_\IZ \times (W^f_\IT)^{[2,\infty)}$ endowed with its natural product $\sigma$-algebra $\cA_1$. We write $Y_\point, Z_\point$ and $\wt{Y}_\point^k$, $k \ge 2$, for the canonical coordinate processes on $\Omega_1$, as well as $X_\point = (Y_\point,Z_\point)$. The probability $Q_1$ is constructed as follows.
\begin{align}
&\mbox{The law of $X_{\cdot \wedge D_1}$ under $Q_1$ coincides with $P_{q_{z_0}}[X_{\cdot \wedge D_1} \in dw]$}\,.\label{2.9}
\\[2ex]
&\mbox{The conditional law $Q_1[X_{(D_1 + \cdot) \wedge R_2} \in dw$, $(\wt{Y}_0^2, Z_{R_2}) \in d\wt{x} \,|\,X_{\cdot \wedge D_1}]$} \label{2.10}
\\
&\mbox{equals} \;P_{X_{D_1}} [(X_{\cdot \wedge R_1}) \in dw \,|\,X_{R_1} = x^\prime] \,\rho_{X_{D_1}} (dx^\prime, d\wt{x})\,.\nonumber
\end{align}

\pagebreak\n
With (\ref{2.9}), (\ref{2.10}) the law of $(X_{\cdot \wedge R_2})$, $\wt{Y}_0^2$ under $Q_1$ is specified. We then proceed as follows.
\begin{align}
&\mbox{Conditionally on $(X_{\cdot \wedge R_2})$, $\wt{Y}_0^2$, the law of $X_{(R_2 + \cdot) \wedge D_2}$ under $Q_1$ is} \label{2.11}
\\[-1ex]
& P_{X_{R_2}}[(X_{\cdot \wedge  T_{\wt{B}}}) \in dw]\,.  \nonumber
\\[2ex]
&\mbox{If $\wt{Y}^2_0 = Y_{R_2} \big( = \pi_\IT (X_{R_2})\big)$, then $\wt{Y}_\point^2 = Y_{(R_2 + \cdot) \wedge D_2}$, $Q_1$-a.s.}\,. \label{2.12}
\\[2ex]
&\mbox{If $\wt{Y}_0^2 \not= Y_{R_2}$, then conditionally on $X_{\cdot \wedge D_2}$, $\wt{Y}_0^2$, the law of $\wt{Y}_\point^2$ under $Q_1$ is} \label{2.13}
\\[-1ex]
&P_{(\wt{Y}^2_0, Z_{R_2})} [Y_{\cdot \wedge T_{\wt{B}}} \in d w^\prime \,|\,Z_{\cdot \wedge T_{\wt{B}}} = w(\cdot)], \;\mbox{where} \; w(\cdot) = Z_{(R_2 + \cdot) \wedge D_2} ( = \pi_\IZ (X_{(R_2 + \cdot) \wedge D_2})\,.\nonumber
\end{align}

\n
The above steps specify the law of $(X_{\cdot \wedge D_2}, \wt{Y}_\point^2)$ under $Q_1$. We then proceed using the kernel of the last line of (\ref{2.10}) with $X_{D_2}$ in place of $X_{D_1}$ to specify the conditional law under $Q_1$ of $(X_{\cdot \wedge R_3})$, $\wt{Y}_0^3$, given $X_{\cdot \wedge D_2}$, $\wt{Y}_\point^2$ and so on and so forth to construct the full law $Q_1$.

\medskip
With this construction the claim (\ref{2.3}) follows directly from (\ref{2.6}). Then (\ref{2.5}) follows from (\ref{2.8}) and the statements (\ref{2.10}), (\ref{2.12}) and their iteration for arbitrary $k \ge 2$. The proof of (\ref{2.4}) is similar to the proof of (3.22) in Proposition 3.3 of \cite{Szni09a}.
\end{proof}

\begin{remark}\label{rem2.3} \rm As a direct consequence of Proposition \ref{prop2.2} we see that for $\alpha > 0$, $N \ge 1$, $K$ as below (\ref{1.24}) and $z_0 \in I$,
\begin{equation}\label{2.14}
Q_1 \Big[\bigcup\limits_{1 \le k \le K} \;X_{[R_k,D_k]} = \bigcup\limits_{1 \le k \le K} \;{\rm range} \;\wt{X}_\point^k\Big] \stackrel{(\ref{2.5})}{\ge} 1 - c \,\alpha \;\dis\frac{N^{-3d}}{h_N} \;.
\end{equation}

\n
This estimate will be used in the next section. \hfill $\square$
\end{remark}

\section{Domination by iid excursions}
\setcounter{equation}{0}

In this section we carry out step c) of the outline below Theorem \ref{theo1.1}. We construct a coupling $Q_2$ of $X_\point$, $\wt{X}^k_\point$, $k \ge 1$, see the previous section, with a collection $X_\point^{\prime k}$, $k \ge 1$, of iid excursions having same distribution as $X_{\cdot \wedge T_{\wt{B}}}$ under $P_q$, (the ``special'' excursions), in such a fashion that the trace on $A$, cf.~(\ref{1.24}), of $X_{[0,D_K]}$ is with high probability dominated by the trace on $A$ of the union of the ranges of $X_\point^{\prime k}$, with $k \le K^\prime$ and $K^\prime$ ``slightly'' bigger than $K$, cf.~(\ref{1.26}). This is carried out in Proposition \ref{prop3.1}. As mentioned in the introduction the interest of this coupling is that, roughly speaking, the iid ``special'' excursions $X_\point^{\prime k}$, $k \ge 1$, bring us closer to random interlacements (especially once we carry out a Poissonization step in the next section). The idea for the construction of the coupling is to introduce iid sequences of excursions $\zeta_i^{(z_1,z_2)}$, $i \ge 1$, where $(z_1,z_2)$ varies over $\{r_N, -r_N\} \times \{h_N, -h_N\}$ and classifies the possible entrance and exit levels of the excursions respectively distributed as $X_{\cdot \wedge T_{\wt{B}}}$ under $P_{z_1,z_2}$, cf.~(\ref{1.25}). The sequence $\wt{X}_\point^k$, $k \ge 1$, is in essence realized by picking for each $k$ an excursion of type $(z_1,z_2)$ with $z_1 = Z_{R_k}$ and $z_2 = Z_{D_k}$, whereas the sequence $X_\point^{\prime k}$, $k \ge 1$, is realized by selecting for each $k$ an excursion of type $(z_1,z_2)$ with $z_1 = Z^\prime_{R,k}$, $z_2 = Z^\prime_{D,k}$, where $(Z^\prime_{R,k}, Z^\prime_{D,k})$, $k \ge 1$, is an independent iid sequence with same law as $(Z_{R_1}, Z_{D_1})$ under $P_q$. The domination of the union of the ranges of the $\wt{X}_\point^k, k \le K$, in terms of the union of the ranges of the $X_\point^{\prime k}, k \le K^\prime$, then relies on large deviation estimates for the empirical measure of the $(Z_{R_k},Z_{D_k})$ under $P_{q_{z_0}}$ and of the empirical measure of the iid variables $(Z^\prime_{R,k}, Z^\prime_{D,k})$. The excursion $\wt{X}_\point^1$ requires a special treatment due to its atypical starting height $z_0 \in I$, which possibly differs from $\pm r_N$.

\medskip
The notation $\cT_F$ for $F \subseteq E$ has been introduced below (\ref{1.1}), and $K^\prime$ is defined in (\ref{1.26}).

\begin{proposition}\label{prop3.1} $(\alpha > 0, v > (d+1) \alpha)$

\medskip
For $N \ge c_1(\alpha,v)$, $z_0 \in I$, one can construct on an auxiliary space $(\Omega_2, \cA_2)$ a coupling $Q_2$ of $X_\point$, $\wt{X}_\point^k$, $k \ge 1$, under $Q_1$ and of $X_\point^{\prime k}$, $k \ge 1$, iid $\cT_{\wt{B}}$-valued variables with same distribution as $X_{\cdot \wedge T_{\wt{B}}}$ under $P_q$, so that:
\begin{equation}\label{3.1}
Q_2 \Big[\bigcup\limits_{1 \le k \le K} X_{[R_k,D_k]} \subseteq  \bigcup\limits_{1 \le k \le K^\prime} \;{\rm range} \;X^{\prime k}_\point\Big] \ge 1 - c \,N^{-3d}\,.
\end{equation}
\end{proposition}

\begin{proof}
We introduce the space $\Gamma$ of ``excursion types'':
\begin{equation}\label{3.2}
\Gamma = \{r_N, - r_N\} \times \{h_N, - h_N\}\,,
\end{equation}

\medskip\n
and for $\gamma = (z_1,z_2) \in \Gamma$ write $P_\gamma$ in place of $P_{z_1,z_2}$, cf.~(\ref{1.25}).

\medskip
We consider an auxiliary probability space $(\Sigma, \cB, M)$ endowed with the following collection of variables and processes:
\begin{align}
&\mbox{the variables $(Z_{R,k}, Z_{D,k})$, $k \ge 1$, with values in $\{z_0\} \times \{h_N,-h_N\}$, when $k = 1$,}\label{3.3}
\\
&\mbox{and in $\Gamma$, when $k \ge 2$, distributed as $(Z_{R_k}, Z_{D_k})$, $k \ge 1$, under $P_{q_{z_0}}$}\,, \nonumber
\\[2ex]
&\mbox{the iid variables $(Z^\prime_{R,k}, Z^\prime_{D,k})$, $k \ge 1$, with same distribution as $(Z_{R_1},Z_{D_1})$}\label{3.4}
\\
&\mbox{under $P_q$}\,, \nonumber
\\[2ex]
&\mbox{the independent $\cT_{\wt{B}}$-valued $\zeta_i^\gamma(\cdot)$, $i \ge 1$, $\gamma \in \Gamma$, such that $\zeta^\gamma_i(\cdot)$ is distributed as}\label{3.5}
\\
&\mbox{$X_{\cdot \wedge T_{\wt{B}}}$ under $P_\gamma$}\,, \nonumber
\\[2ex]
&\mbox{the iid $\cT_{\wt{B}}$-valued $\overline{\zeta}_i(\cdot)$, $i \ge 1$, with same distribution as $X_{\cdot \wedge T_{\wt{B}}}$ under $P_q$,}\label{3.6}
\intertext{and so that}
&\mbox{the above collections in (\ref{3.3}) - (\ref{3.6}) are mutually independent.} \label{3.7}
\end{align}
We then introduce the $\Gamma$-valued processes $\gamma_k, k \ge 1$, and $\gamma^\prime_k, k \ge 1$, via:
\begin{align}
&\gamma_1 = (r_N, Z_{D,1}) \;\mbox{and} \;\gamma_k = (Z_{R,k},Z_{D,k}), \;\mbox{for} \; k \ge 2\,, \label{3.8}
\\[2ex]
&\gamma^\prime_k = (Z_{R,k}^\prime, \,Z^\prime_{D,k}), \; k \ge 1\,. \label{3.9}
\end{align}

\medskip\n
The definition of $\gamma_1$ in (\ref{3.8}) is somewhat arbitrary as a consequence of the special role of the starting point $z_0 \in I$ of the walk. We also consider the counting functions:
\begin{equation}\label{3.10}
\begin{array}{l}
N_k(\gamma) = \big | \{j \in [2,k]; \,\gamma_j = \gamma\}\big|, 
\\[1ex]
 N^\prime_k(\gamma) = |\{j \in [1,k]; \;\gamma_j^\prime = \gamma\}|, \;\mbox{for} \;\gamma \in \Gamma, k \ge 1\,.
 \end{array}
\end{equation}

\medskip\n
We will now introduce processes $\ovxkpoint, k \ge 1$, on $(\Sigma, \cB, M)$ which have same law as $\wt{X}_\point^k, k \ge 1$, under $Q_1$, cf.~Proposition \ref{prop2.2}. To this effect we define:
\begin{equation}\label{3.11}
\mbox{$i_0 = \inf\big\{i \ge 1; \,\ov{\zeta}_i(\cdot)$ enters $\IT \times \{z_0\}$ before exiting $\wt{B}$ through $\IT \times \{Z_{D_,1}\}\big\}$}\,,
\end{equation}

\n
where we note that since $P_q [H_{\IT \times \{z_0\}} < T_{\wt{B}}$ and $Z_{T_{\wt{B}}} = z] > 0$, for $z = \pm h_N$, one has $i_0 < \infty$, $M$-a.s., thanks to (\ref{3.6}), (\ref{3.7}). Further we observe that
\begin{align}
&\mbox{conditionally on $(Z_{R,k}, Z_{D,k})$, $k \ge 1$, the processes $\ov{\zeta_{i_0}} (H_{\IT \times \{z_0\}} + \, \cdot)$ and} \label{3.12}
\\
&\mbox{$\zeta^{\gamma_k}_{N_k(\gamma_k)}(\cdot), k \ge 2$, are independent and respectively distributed as $X_{\cdot \wedge T_{\wt{B}}}$} \nonumber
\\
&\mbox{under $P_{Z_{R,1},Z_{D,1}}$ and $P_{Z_{R,k},Z_{D,k}}, k \ge 2$}\,.\nonumber
\end{align}

\medskip\n
Taking into account (\ref{2.4}) and (\ref{3.3}) we have thus obtained that defining
\begin{align}
&\mbox{$\ov{X}_\point{\hspace{-0.7ex}^k}= \ov{\zeta}_{i_0}(H_{\IT \times \{z_0\}} + \cdot)$, when $k = 1$, $\zeta^{\gamma_k}_{N_k(\gamma_k)}(\cdot)$, when $k \ge 2$}\,, \label{3.13}
\intertext{one finds that}
&\mbox{$(\ovxkpoint)_{k \ge 1}$ under $M$ has same distribution as $(\wt{X}_\point^k)_{k \ge 1} $ under $Q_1$}\,.\label{3.14}
\end{align}

\medskip\n
In a similar fashion we also define the processes
\begin{equation}\label{3.15}
\wh{X}_\point^k = \zeta^{\gamma^\prime_k}_{N^\prime_k(\gamma_k^\prime)}(\cdot), \;k \ge 1\,.
\end{equation}

\n
Observe that conditionally on $\gamma_k^\prime$, $k \ge 1$, the $\wh{X}_\point^k$, $k \ge 1$, are independent with respective distribution that of $\Xdotwtb$ under $P_{\gamma^\prime_k}$ or equivalently under $P_q[\cdot \,|(Z_{R_1},Z_{D_1}) = \gamma^\prime_k]$. Since the $\gamma^\prime_k$, $k \ge 1$, are iid $\Gamma$-valued variables with same distributions as $(Z_{R_1},Z_{D_1})$ under $P_q$, cf.~(\ref{3.4}), (\ref{3.9}), it follows that
\begin{align}
&\mbox{$\wh{X}_\point^k, k \ge 1$, are iid $\cT_{\wt{B}}$-valued with same distribution as $\Xdotwtb$ under $P_q$}\,,\label{3.16}
\\[-1ex]
&\mbox{and they are independent from the collection $\ov{\zeta}_i(\cdot), i \ge 1$}\,.\nonumber
\end{align}

\n
We recall the definition of $\delta$ in (\ref{1.26}) and then set 
\begin{align}
& \wh{K} = \Big[\Big( 1 + \mbox{\f $\dis\frac{\delta}{5}$}\Big) \,\alpha N^d/h_N\Big], \;\mbox{as well as} \label{3.17}
\\[2ex]
& \mbox{$X_\point^{\prime k} = \wh{X}_\point^k$, when $1 \le k \le \wh{K}$, $\ov{\zeta}_{k - \wh{K}}(\cdot)$, when $k > \wh{K}$}\,. \label{3.18}
\end{align}

\n
It now follows from (\ref{3.6}), (\ref{3.16}) that
\begin{equation}\label{3.19}
\mbox{$X_\point^{\prime k}, k \ge 1$, are iid with same distribution as $\Xdotwtb$ under $P_q$}\,.
\end{equation}

\n
We then introduce the ``good event'':
\begin{equation}\label{3.20}
\mbox{$\cG = \{i_0 \le K^\prime - \wh{K}$ and for each $\gamma \in \Gamma$, $N_K(\gamma) \le N^\prime_{\wh{K}}(\gamma)\}$}\,.
\end{equation}

\n
The interest of this definition stems from the fact that on $\cG$ 
\begin{equation*}
{\rm range} \; \ov{X}_\point{\hspace{-0.7ex}^1} \stackrel{(\ref{3.13})}{\subseteq}\; {\rm range}\; \ov{\zeta}_{i_0} \stackrel{(\ref{3.18})}{\subseteq} \bigcup\limits_{\wh{K} < k \le K^\prime} \;{\rm range} \;X_\point^{\prime k} \,,
\end{equation*}
as well as
\begin{equation*}
\begin{array}{l}
\bigcup\limits_{2 \le k \le K} {\rm range} \,\ovxkpoint \stackrel{(\ref{3.13})}{=} \bigcup\limits_{\gamma \in \Gamma} \;\bigcup\limits_{i \le N_K(\gamma)} {\rm range} \,(\zeta^\gamma_i) \subseteq \bigcup\limits_{\gamma \in \Gamma} \;\bigcup\limits_{i \le N^\prime_{\wh{K}}(\gamma)} {\rm range} \,(\zeta_i^\gamma) \stackrel{(\ref{3.13}),(\ref{3.18})}{\subseteq}
\\
\bigcup\limits_{1 \le k \le \wh{K}} {\rm range} \,X^{\prime k}_\point \,.
\end{array}
\end{equation*}
As a result we see that
\begin{equation}\label{3.21}
M \Big[ \bigcup\limits_{1 \le k \le K} {\rm range} \,\ovxkpoint \subseteq \bigcup\limits_{1 \le k \le K^\prime} {\rm range} \,X_\point^{\prime k}\Big] \ge M(\cG)\,.
\end{equation}

\n
We will now explain why Proposition \ref{prop3.1} follows once we show that
\begin{equation}\label{3.22}
\mbox{for}\; N \ge c(\alpha,v), \quad M(\cG) \ge 1 - c\,N^{-3d} \,.
\end{equation}

\n
For this purpose we note that $(\Omega_1,\cA_1)$, see above (\ref{2.9}), is a standard measurable space, cf.~\cite{IkedWata89}, p.~13. The $(\wtxkpoint)_{k \ge 1}$ after modification on a $Q_1$-negligible set can be viewed as $(\wt{\Omega}, \wt{\cA})$-valued variables where $\wt{\Omega}$ stands for $\cT_E^{[1,\infty)}$, with $\cT_E$ the countable space defined below (\ref{1.1}), and where $\wt{\cA}$ denotes the canonical product $\sigma$-algebra, so that $(\wt{\Omega}, \wt{\cA})$ is also a standard measurable space. With Theorem 3.3, p.~15 of \cite{IkedWata89} and its corollary we can find a probability kernel $q(\wt{\o},d\o_1)$ from $(\wt{\Omega}, \wt{\cA})$ to $(\Omega_1,\cA_1)$ such that for any bounded $\cA_1$-measurable function $f$ on $\Omega_1$, $h((\wt{X}_\point^k)_{k \ge 1})$, where $h(\wt{\o}) = \int_{\Omega_1} f(\o_1) \,q(\wt{\o}, d \o_1)$, is a version of $E^{Q_1}[f \,|\,(\wtxkpoint)_{k \ge 1}]$ such that for a.e. $\wt{\o}$ relative to the $\cA_1$-law of $(\wtxkpoint)_{k \ge 1}$, $q(\wt{\o}, \cdot)$ is supported on the fiber $\{\o_1 \in \Omega_1$; $(\wtxkpoint)_{k \ge 1}(\o_1) = \wt{\o}\}$.

\medskip
We can thus define $\Omega_2 = \Sigma \times \Omega_1$, $\cA_2 = \cB \otimes \cA_1$, and $Q_2$ the semiproduct of $M$ with the kernel $q ((\ovxkpoint)_{k \ge 1}, d \o_1)$. Since $(\ovxkpoint)_{k \ge 1}$ under $M$ has same law as $(\wtxkpoint)_{k \ge 1}$ under $Q_1$, cf.~(\ref{3.14}), it follows that $Q_2$-a.s. $\wtxkpoint = \ovxkpoint$, for $k \ge 1$, and $X_\point$, $(\wtxkpoint)_{k \ge 1}$ has the same law under $Q_2$ (on the enlarged space $\Omega_2)$ as under $Q_1$. The claim (\ref{3.1}) then follows from (\ref{3.21}), (\ref{3.22}) together with (\ref{2.14}).

\medskip
We now turn to the proof of (\ref{3.22}). We begin with an upper bound on $M[i_0 > K^\prime - \wh{K}]$. For $z \in \{h_N, - h_N\}$, we have the identity (with hopefully obvious notation):
\begin{equation}\label{3.23}
\begin{array}{l}
P_q \big[H_{\IT \times \{z_0\}} < T_{\wt{B}} \;\mbox{and} \;X_{T_{\wt{B}}} \in \IT \times \{z\}\big] = P_q [H_{\IT \times \{z_0\}} < T_{\wt{B}}] \,P_{z_0}^{\IZ} [X_{T_{\wt{I}}} = z] =
\\[2ex]
\Big(\fr \;\dis\frac{h_N - r_N}{h_N + z_0} + \fr \;\dis\frac{h_N - r_N}{h_N - z_0}\Big) \;\dis\frac{z + z_0}{2 z} = \fr \;\dis\frac{h_N - r_N}{h^2_N - z_0^2} \;|z + z_0| \ge 
\fr \;\dis\frac{h_N - r_N}{h_N + r_N} \ge \frvier\;, 
\\[3ex]
\mbox{for $N \ge c$} \,.
\end{array}
\end{equation}

\n
As a result of (\ref{3.6}), (\ref{3.11}), we thus find that for $N \ge c(\alpha, v)$, 
\begin{equation}\label{3.24}
M [i_0 > K^\prime - \wh{K}] \le \Big(\mbox{\f $\dis\frac{3}{4}$}\Big)^{K^\prime - \wh{K}} \le e^{-c(\alpha, v)K} \,.
\end{equation}

\n
The next step in the proof of (\ref{3.22}) is the derivation of an upper bound on $M[N_K(\gamma) > N^\prime_{\wh{K}}(\gamma)]$, for $\gamma \in \Gamma$. We introduce the probabilities
\begin{equation}\label{3.25}
\begin{array}{l}
p_N = P^{\IZ}_{r_N} \big[H_{h_N} < H_{-h_N}\big] = P^{\IZ}_{-r_N} \big[H_{-h_N} < H_{h_N}\big] = \dis\frac{h_N + r_N}{2h_N} \,, \;\mbox{and}
\\[1ex]
q_N = 1 - p_N = \dis\frac{h_N - r_N}{2h_N}\;, \;\mbox{so that} \;\Big|p_N - \fr \Big| = \Big| q_N - \fr \Big| = \fr \;\dis\frac{r_N}{h_N}\;.
\end{array}
\end{equation}

\n
With (\ref{3.4}), we see that for $\gamma = (z_1,z_2) \in \Gamma$, and $k \ge 1$,
\begin{equation}\label{3.26}
M \big[(Z^\prime_{R,k}, Z^\prime_{D,k}) = \gamma \big] = \fr \;p_N\,1\{z_1 z_2 > 0\} + \fr \;q_N \,1\{z_1 z_2 < 0\} \stackrel{\rm def}{=} p(\gamma) \,.
\end{equation}

\n
Then with the help of a Cramer-type exponential bound it follows that for $\rho > 0$, $\gamma \in \Gamma$,
\begin{equation*}
M \Big[N^\prime_{\wh{K}} (\gamma) \le \Big(\frvier - \mbox{\f $\dis\frac{\delta}{100}$}\Big) \,\wh{K}\Big] \le \exp \Big\{ \wh{K} \Big[ \Big( \frvier - \mbox{\f $\dis\frac{\delta}{100}$}\Big) \,\rho + \log \big(1 - (1-e^{-\rho})\,p(\gamma)\big)\Big]\Big\}\,.
\end{equation*}

\n
Hence for $N \ge c(\alpha, v)$, (ensuring in particular $p(\gamma) \ge \frac{1}{4} - \frac{\delta}{200}$, for all $\gamma \in \Gamma$), and $\rho = c^\prime(\alpha, v)$ small enough, the above inequality yields that
\begin{equation}\label{3.27}
M \Big[N^\prime_{\wh{K}} (\gamma) \le \Big(\frvier - \mbox{\f $\dis\frac{\delta}{100}$}\Big) \,\wh{K}\Big] \le e^{-c(\alpha,v)\wh{K}} \le e^{-c(\alpha,v)K}\,.
\end{equation}

\n
The last (and main) step in the proof of (\ref{3.22}) is the derivation of an upper bound on $M[N_K(\gamma) > (\frac{1}{4} + \frac{\delta}{100}) \,K]$. We will rely on large deviation estimates for the empirical measure of $(Z_{R,k},Z_{D,k})$, $k \ge 2$, cf.~(\ref{3.3}). In essence, as we will see below, this boils down to large deviation estimates on the pair empirical distribution of a Markov chain on $\{1, -1\}$, which at each step remains at the same location with probability $p_N$, (close to $\frac{1}{2}$, cf.~(\ref{3.25})), and changes location with probability $q_N = 1-p_N$. The transition probabilities of this Markov chain depend on $N$, and to derive the relevant large deviation estimates with uniformity over $N$, we rely on super-multiplicativity, cf.~Lemma 6.3.1 of \cite{DembZeit98}, p.~273.

\medskip
In view of (\ref{3.3}), $M_\point$-a.s., for $k \ge 2$, $Z_{D,k-1}$ and $Z_{R,k}$ have same sign. We denote with $\phi$ the bijective map from $\Gamma$ onto $\wt{\Gamma} \stackrel{\rm def}{=} \{1,-1\}^2$, defined by $\phi(\gamma) = ({\rm sign}(z_1)$, sign$(z_2)$, for $\gamma \in \Gamma$. We consider the $\wt{\Gamma}$-valued stochastic process, cf.~(\ref{3.9}),
\begin{equation*}
\begin{split}
\wt{\gamma}_k = \phi(\gamma_k) & = \big(1, {\rm sign}(Z_{D,1})\big), \;\mbox{$k = 1$} \,,
\\
&\stackrel{\rm a.s.}{=}  \big({\rm sign}(Z_{D,k-1}), \;{\rm sign}(Z_{D,k})\big), \;\mbox{for $k \ge 2$}\,.
\end{split}
\end{equation*}

\n
Note that under $M$, $({\rm sign}(Z_{D,k}))_{k \ge 1}$, has the same law as $(Z_{D_k}/h_N)_{k \ge 1}$, under $P_{q_{z_0}}$, cf.~(\ref{3.3}), which is a Markov chain on $\{1,-1\}$, which at each step has a transition probability $p_N$ to remain at the same location and $q_N$ to change location, as well as an initial distribution (at time $1$) $P_{q_{z_0}}[H_{\IT \times \{h_N\}} < H_{\IT \times \{-h_N\}}] = \frac{h_N + z_0}{2h_N}$ to be at $1$ and $\frac{h_N - z_0}{2 h_N}$ to be at $-1$. This chain on $\{1,-1\}$ induces a Markov chain on $\wt{\Gamma} = \{1,-1\}^2$ by looking at consecutive positions of the original chain, so that when located in $\wt{\gamma} =(\wt{\gamma}\,^1, \wt{\gamma}\,^2) \in \wt{\Gamma}$, the induced  chain jumps with probability $p_N$ to $(\wt{\gamma}\,^2, \wt{\gamma}\,^2)$ and $q_N$ to $(\wt{\gamma}\,^2, - \wt{\gamma}\,^2)$. We denote with $\wt{R}_{\wt{\gamma}}$, for $\wt{\gamma} \in \wt{\Gamma}$, the canonical law on $\wt{\Gamma}\,^{\IN}$ of this chain starting at $\wt{\gamma}$, and with $U_m$, $m \ge 0$, its canonical process. This is an irreducible chain on $\wt{\Gamma}$ and 
\begin{align}
&\mbox{$\wt{\gamma}_k$, $k \ge 1$, under $M$ has same law as $U_{k-1}$, $k \ge 1$, under $\wt{R}_{\wt{\kappa}}$, where}\label{3.28}
\\
&\mbox{$\wt{\kappa}$ stands for the initial distribution $\mbox{\f $\dis\frac{h_N + z_0}{2h_N}$} \;\delta_{(1,1)} + \mbox{\f $\dis\frac{h_N - z_0}{2h_N}$} \;\delta_{(1,-1)}$}\,. \nonumber
\end{align}

\n
Using sub-additivity, see \cite{DembZeit98}, p.~273 and 275, we see that for $N \ge 1$, 
\begin{equation}\label{3.29}
\inf\limits_{\wt{\sigma} \in \wt{\Gamma}} \wt{R}_{\wt{\sigma}} \;\Big[\mbox{\f $\dis\frac{1}{n}$} \;\dsl^n_{m=1} \;1 \{U_m = \wt{\gamma}\} \ge v\Big] \le e^{-n \Psi_N (\wt{\gamma}, v)}, \;\mbox{for} \;n \ge 1, \wt{\gamma} \in \wt{\Gamma}, 0 < v < 1\,,
\end{equation}

\n
where thanks to the fact that the chain on $\wt{\Gamma}$ describes the evolution of pairs of consecutive positions of the chain on $\{1, -1\}$ mentioned above, see Theorem 3.1.13, p.~79 of \cite{DembZeit98}, we have set for $\wt{\gamma} \in \wt{\Gamma}$, $0 < v < 1$,
\begin{equation}\label{3.30}
\Psi_N(\wt{\gamma},v) = \inf\{H_{2,N}(\mu); \;\mbox{$\mu$ probability on $\wt{\Gamma}$ with $\mu(\{\wt{\gamma}\}) \ge v\}$}\,,
\end{equation}
and for $\mu$ probability on $\wt{\Gamma}$
\begin{equation}\label{3.31}
\begin{split}
H_{2,N} (\mu)  =&\; \infty, \;\mbox{when the two marginals of $\mu$ are different}\,,
\\
 = &\;\mu(1,1) \log \Big(\mbox{\f $\dis\frac{\mu(1|1)}{p_N}$}\Big) + \mu(-1,-1) \log \Big(\mbox{\f $\dis\frac{\mu(-1|-1)}{p_N}$}\Big) \; +
\\[1ex]
& \;  \mu(1,-1) \log \Big(\mbox{\f $\dis\frac{\mu(-1|1)}{q_N}$}\Big) + \mu(-1,1) \log \Big(\mbox{\f $\dis\frac{\mu(1|-1)}{q_N}$}\Big),\;\mbox{otherwise}\,,
\end{split}
\end{equation}

\medskip\n
where we wrote $\mu(i,j)$ in place of $\mu(\{i,j\})$, for $i,j \in \{1, -1\}$, and $\mu(j | i)$ for the $\mu$-conditional probability that the second coordinates equals $j$ given that the first coordinate equals $i$.

\medskip
We then introduce $\Psi_\infty$ and $H_{2,\infty}$ as in (\ref{3.30}), (\ref{3.31}) replacing $p_N$ and $q_N$ with $\frac{1}{2}$. In view of the last line of (\ref{3.25}) we see that for $N \ge c$, for any probability $\mu$ on $\wt{\Gamma}$
\begin{equation}\label{3.32}
\begin{array}{l}
\mbox{the finiteness of $H_{2,N}(\mu)$ and $H_{2,\infty}(\mu)$ are equivalent and when this holds,}
\\
|H_{2,N}(\mu) - H_{2,\infty}(\mu) \,|\,\le c \;\mbox{\f $\dis\frac{r_N}{h_N}$}\;.
\end{array}
\end{equation}

\n
The non-negative function $H_{2,\infty}$ is lower semi-continuous relative to weak convergence, cf. \cite{DembZeit98}, p.~79, and only vanishes at the equidistribution on $\wt{\Gamma}$. As a result $\Psi_\infty (\wt{\gamma}, \;\frac{1}{4} + \frac{\delta}{200}) > 0$, for each $\wt{\gamma} \in \Gamma$, so that with (\ref{3.29}), (\ref{3.32}), when $N \ge c (\alpha, v)$ one finds
\begin{equation}\label{3.33}
\inf\limits_{\wt{\sigma} \in \wt{\Gamma}} \;\wt{R}_{\wt{\sigma}} \Big[ \mbox{\f $\dis\frac{1}{n}$} \;\dsl^n_{m=1} \;1 \{U_m = \wt{\gamma}\} \ge \frvier + \mbox{\f $\dis\frac{\delta}{200}$}\Big] \le e^{-n c^\prime (\alpha, v)}, \;\mbox{for all $\wt{\gamma} \in \wt{\Gamma}$, $n \ge 1$}\,.
\end{equation}

\n
Since $\inf_{\wt{\sigma}, \wt{\gamma} \in \wt{\Gamma}} \wt{R}_{\wt{\sigma}} [U_2 = \wt{\gamma}] \ge c > 0$, it follows that for $\wt{\gamma} \in \wt{\Gamma}$, $n \ge 1$
\begin{equation}\label{3.34}
\begin{array}{l}
\sup\limits_{\wt{\sigma} \in \wt{\Gamma}} \;\wt{R}_{\wt{\sigma}} \Big[\dsl^n_{m=1} \;1\{U_m = \wt{\gamma}\} \ge \Big(\frvier + \mbox{\f $\dis\frac{\delta}{100}$}\Big) n\Big] \le
\\
 \mbox{\f $\dis\frac{1}{c}$} \;\inf\limits_{\wt{\sigma} \in \wt{\Gamma}} \;\wt{R}_{\wt{\sigma}} \Big[\dsl^{n+2}_{m=1} \;1\{U_m = \wt{\gamma}\} \ge \Big(\frvier + \mbox{\f $\dis\frac{\delta}{100}$}\Big) n\Big] \stackrel{(\ref{3.33})}{\le}
 \\
 c^\prime \exp\{ - (n+2) \,c^\prime(\alpha, v)\}, \;\mbox{as soon as $n  \Big(\frvier + \mbox{\f $\dis\frac{\delta}{100}$}\Big) \ge (n+2) \; \Big(\frvier + \mbox{\f $\dis\frac{\delta}{200}$}\Big)$} \,.
\end{array}
\end{equation}

\n
Since $\phi$ is a bijection between $\Gamma$ and $\wt{\Gamma}$ and $\wt{\gamma}_k = \phi (\gamma_k)$, we can now deduce from (\ref{3.28}) and (\ref{3.34}) with $n = K-1$ that for $N \ge c(\alpha,v)$, 
\begin{equation}\label{3.35}
M\Big[N_K(\gamma) \ge  \Big(\frvier + \mbox{\f $\dis\frac{\delta}{100}$}\Big)  K\Big] \le c\,e^{-c(\alpha,v)K}, \;\mbox{for $\gamma \in \Gamma$}\,.
\end{equation}

\n
For large $N$, one has $(\frac{1}{4} + \frac{\delta}{100}) K < (\frac{1}{4} - \frac{\delta}{100}) \wh{K}$, cf.~(\ref{3.17}), and hence with (\ref{3.27}), (\ref{3.35}), for $N \ge c(\alpha,v)$:
\begin{equation}\label{3.36}
M[N_K(\gamma) \ge N^\prime_{\wh{K}}(\gamma)] \le c\,e^{-c^\prime (\alpha,v) K}, \;\mbox{for each $\gamma \in \Gamma$}\,.
\end{equation}

\n
Together with (\ref{3.24}) this proves (\ref{3.22}) and concludes the proof of Proposition \ref{prop3.1}.
\end{proof}

\begin{remark}\label{rem3.2} \rm
Although we will not need this fact let us mention that $H_{2,N}$ in (\ref{3.31}) is a non-negative lower continuous function for the weak convergence. Moreover it vanishes at the unique probability on $\wt{\Gamma}$, for which the first coordinate is equidistributed and conditionally on the first coordinate the second coordinate coincides with the first coordinate with probability $p_N$ (and differs with probability $q_N$). This last feature follows from the relative entropy interpretation of $H_{2,N}$, cf.~\cite{DembZeit98}, p.~79.  \hfill $\square$
\end{remark}

\section{Poissonization}
\setcounter{equation}{0}

This section carries out step d) of the outline below Theorem \ref{theo1.1}. We construct a coupling $Q_3$ of $X_\point$, $X_\point^{\prime k}$, $k \ge 1$, under $Q_2$ with an independent Poisson variable $J^\prime$ of parameter $(1 + \frac{3}{5}\,\delta) \,\alpha \,\frac{N^d}{h_N}$. This enables to define a Poisson point measure $\mu^\prime$ on $\cT_{\wt{B}}$, cf.~(\ref{4.1}) and below (\ref{1.1}) for the definition of $\cT_{\wt{B}}$, such that the union of the ranges of trajectories in the support of $\mu^\prime$ with high probability contains the trace on $A$ of $X_{[0,D_K]}$. 

\medskip
We thus consider, cf.~Proposition \ref{prop3.1}, $N \ge c_1 (\alpha,v)$, $z_0 \in I$ and $\Omega_3 = \Omega_2 \times \IN$ endowed with the product $\sigma$-algebra $\cA_3 = \cA_2 \times \cP(\IN)$, where $\cP(\IN)$ stands for the collection of subets of $\IN$, and the probability $Q_3$ product of $Q_2$ with the Poisson law of parameter $(1 + \frac{3}{5}\,\delta) \,\alpha \,\frac{N^d}{h_N}$. We denote with $J^\prime$ the $\IN$-valued coordinate which is Poisson distributed. The definition of $A$ appears below (\ref{1.24}).

\begin{proposition}\label{prop4.1} $(\alpha > 0$, $v > (d+1) \,\alpha$, $0 < \ve < 1)$

\medskip
For $N \ge c(\alpha,v, \ve)$ and $z_0 \in I$, the random point measure on $\cT_{\wt{B}}$ defined by
\begin{equation}\label{4.1}
\mu^\prime = \dsl_{1 \le k \le J^\prime} \;\delta_{X_{H_A + \cdot}^{\prime k}} \,1\{{\rm range} \,X_\point^{\prime k} \cap A \not= \phi\}
\end{equation}

\n
is Poisson with intensity measure $\lambda^\prime \,\kappa^\prime$ on $\cT_{\wt{B}}$, where
\begin{equation}\label{4.2}
\lambda^\prime = \Big(1 + \mbox{\f $\dis\frac{3}{5}$}\,\delta\Big) \,\alpha (d+1) \Big(1 - \mbox{\f $\dis\frac{r_N}{h_N}$}\Big) \;\mbox{and $\kappa^\prime$ is the law of $\Xdotwtb$ under $P_{e_{A,\wt{B}}}$}\,.
\end{equation}

\n
Moreover if one defines the random subset of $A$
\begin{equation}\label{4.3}
\cI^\prime = \bigcup\limits_{w \in {\rm Supp} \,\mu^\prime} \;{\rm range} \,w \cap A\,,
\end{equation}
then one has
\begin{equation}\label{4.4}
Q_3 [X_{[0,D_K]} \cap A \subseteq \cI^\prime] \ge 1 - c\,N^{-3d}\,.
\end{equation}
\end{proposition}

\begin{proof}
Since the $X_\point^{\prime k}$ are iid $\cT_{\wt{B}}$-valued variables, the Poissonian character of $\mu^\prime$ is immediate. It then follows from (\ref{1.13}) and the fact that the $X_\point^{\prime k}$ have same distribution as $\Xdotwtb$ under $P_q$ that $\mu^\prime$ has intensity measure $\lambda^\prime \,\kappa^\prime$ with $\lambda^\prime$ and $\kappa^\prime$ as in (\ref{4.2}).

\medskip
Finally note that on $\{J^\prime \ge K^\prime\}$, $\cI^\prime$ contains $\bigcup_{1 \le k \le K^\prime}$ range $X_\point^{\prime k} \cap A$. Moreover choosing $a = a(\alpha,v)$ small enough one has the exponential bound:
\begin{equation}\label{4.5}
Q_3 [J^\prime < K^\prime] \le \exp \Big\{a K^\prime - \Big(1 + \mbox{\f $\dis\frac{3}{5}$}\,\delta\Big)\,\alpha \;\mbox{\f $\dis\frac{N^d}{h_N}$} \;(1 - e^{-a})\Big\} \le \exp\{ - c(\alpha, v) K\}\,.
\end{equation}

\n
Combined with (\ref{3.1}) and the fact that $X_{[0,D_K]} \cap A \subseteq \bigcup_{1 \le k \le K} X_{[R_k,D_k]} \cap A$, we easily deduce (\ref{4.4}).
\end{proof}

\section{Truncation}
\setcounter{equation}{0}

This section is devoted to step e) of the outline below Theorem \ref{theo1.1}. We construct a coupling $Q_4$ of $X_\point$, $X_\point^{\prime k}$, $k \ge 1$, $\cI^\prime$ under $Q_3$ with a random subset $\cI$ which is the union of the ranges of the trajectories in the support of a suitable Poisson point measure $\mu$ on $\cT_{\wt{C}}$, cf.~(\ref{5.4}), where $\wt{C} = B(0,\frac{N}{4})$. This coupling is such that with high probability $\cI$ contains the trace on $A$ of $X_{[0,D_K]}$. For large $N$ one can view $\wt{C} \cup \partial \wt{C}$ both as a subset of $E$ and $\IZ^{d+1}$, and this makes $\cI$ more convenient than $\cI^\prime$ for the purpose of comparison with random interlacements on $\IZ^{d+1}$, see next section. The main result of this section appears in Proposition \ref{prop5.1}. In the proof we employ the technique of sprinkling introduced in \cite{Szni07a}, and throw in additional trajectories so as to compensate for truncation. This result is very similar to Theorem 3.1 of \cite{Szni08b}, except that in this reference the non-truncated trajectories are $\IZ^{d+1}$-valued whereas they are $\wt{B} \cup \partial \wt{B}$-valued in the present setting. This induces some changes in the proof but the overall spirit remains the same. The main Proposition \ref{prop5.1} then leads to the construction of the desired coupling in Proposition \ref{prop5.4}.

\medskip
We recall the definition of $\wt{C}$ in (\ref{1.27}), and keep the notation of Section 4. We consider an auxiliary probability space $(\Omega_0,\cA_0, Q_0)$ endowed with
\begin{align}
&\mbox{an iid sequence $X_\point^k$, $k \ge 1$, of $\cT_{\wt{B}}$-valued variables with same distributions as}\label{5.1}
\\
&\mbox{$\Xdotwtb$ under $P_q$}\,, \nonumber
\\[2ex]
&\mbox{an independent Poisson variable $J$ with intensity $\Big(1 + \mbox{\f $\dis\frac{4}{5}$}\,\delta\Big) \,\alpha \,\mbox{\f $\dis\frac{N^d}{h_N}$}$}\,. \label{5.2}
\end{align}

\n
This enables to define the Poisson point measure on $\cT_{\wt{C}}$:
\begin{equation}\label{5.3}
\mu = \dsl_{1 \le k \le J} \,\delta_{X^k_{(H_A + \cdot) \wedge T_{\wt{C}}}} \,1 \{{\rm range} \,X_\point^k \cap A \not= \phi\}\,,
\end{equation}

\n
and the same argument as in Proposition \ref{prop4.1} now leads to the fact that for $N \ge c(\ve)$,
\begin{equation}\label{5.4}
\begin{array}{l}
\mbox{$\mu$ has intensity measure $\lambda \kappa$ on $\cT_{\wt{C}}$, where}
\\[1ex]
\mbox{$\lambda = \Big(1 + \mbox{\f $\dis\frac{4}{5}$}\,\delta\Big) \,\alpha (d+1) \Big(1 - \mbox{\f $\dis\frac{r_N}{h_N}$}\Big)$ and $\kappa$ is the law of $X_{\cdot \wedge T_{\wt{C}}}$ under $P_{e_{A,\wt{B}}}$} \,.
\end{array}
\end{equation}
We further introduce
\begin{equation}\label{5.5}
\cI = \bigcup\limits_{w \in {\rm Supp} \,\mu} \;{\rm range} \,w\,.
\end{equation}

\begin{proposition}\label{prop5.1} $(\alpha > 0, v > (d+1) \,\alpha, 0 < \ve < 1)$

\medskip
For $N \ge c(\alpha, v, \ve)$, $z_0 \in I$, there exist random subsets $\cI^*$ and $\ov{\cI}$ of $A$, defined on $(\Omega_3, \cA_3, Q_3)$ of Proposition {\rm \ref{prop4.1}}, such that
\begin{align}
&\cI^\prime = \cI^* \cup \ov{\cI}\,, \label{5.6}
\\[2ex]
&\mbox{$\cI^*$ and $\ov{\cI}$ are independent under $Q_3$}\,, \label{5.7}
\\[2ex]
&Q_3 [\ov{\cI} \not= \phi] \le c\,N^{-3d}\,,\label{5.8}
\\[2ex]
&\mbox{$\cI^*$ is stochastically dominated by $\cI \cap A$}\,.\label{5.9}
\end{align}
\end{proposition}

\begin{proof}
The proof with some modifications follows the same pattern as that of Theorem 3.1 of \cite{Szni08b} and we detail it for the reader's convenience. We consider
\begin{equation}\label{5.10}
M = \big[\exp \{ \sqrt{\log N}\}\big] + 1 \;\mbox{and} \;C = B\Big(0, \Big[\mbox{\f $\dis\frac{N}{4M}$}\Big]\Big) \subseteq \wt{C}\,,
\end{equation}

\n
and from now on assume $N \ge c(\alpha, v,\ve)$ so that Proposition \ref{prop4.1} holds true and
\begin{equation}\label{5.11}
A \subseteq B (0,100 [N^{1-\ve}]) \subseteq C \subseteq B\Big(0, 100 \Big[\mbox{\f $\dis\frac{N}{4M}$}\Big]\Big) \subseteq \wt{C} \subseteq \wt{B}\,.
\end{equation}

\n
We write $\wt{R}_k$ and $\wt{D}_k$, $k \ge 1$, for the successive return times to $A$ and departures from $C$ of a trajectory belonging to $\cT_{\wt{B}}$, just as in (\ref{1.10}) with $B(z)$ and $\wt{B}(z)$ replaced by $A$ and $C$. We then introduce the integer
\begin{equation}\label{5.12}
r =  \Big[\mbox{\f $\dis\frac{16}{\ve}$}\Big] + 1\,,
\end{equation}

\n
as well as the decomposition, see (\ref{4.1}) for the notation:
\begin{equation}\label{5.13}
\begin{split}
\mu^\prime & = \dsl_{1 \le \ell \le r} \;\mu_\ell^\prime + \ov{\mu}, \;\mbox{where}
\\
\mu^\prime_\ell & = 1 \{\wt{D}_\ell < T_{\wt{B}} < \wt{R}_{\ell + 1}\}\,\mu^\prime, \;\mbox{for $\ell \ge 1$, and $\ov{\mu} = 1\{\wt{D}_{r + 1} < T_{\wt{B}}\} \,\mu^\prime$} \,.
\end{split}
\end{equation}

\n
Similarly considering the last return to $A$ before exiting $\wt{C}$, we write $Q_0$-a.s.:
\begin{equation}\label{5.14}
\mu = \dsl_{\ell \ge 1} \;\mu_\ell, \;\mbox{where $\mu_\ell = 1\{\wt{D}_\ell < T_{\wt{C}} < \wt{R}_{\ell + 1}\} \,\mu$, for $\ell \ge 1$} \,.
\end{equation}
Observe that:
\begin{equation}\label{5.15}
\mbox{$\mu^\prime_\ell$, $1 \le \ell \le r$, and $\ov{\mu}$ are independent Poisson measures under $Q_3$}\,,
\end{equation}

\n
and their respective intensity measures on $\cT_{\wt{B}}$ are in the notation of (\ref{4.2}):
\begin{equation}\label{5.16}
\begin{split}
\nu^\prime_\ell & = \lambda^\prime \,1\{\wt{D}_\ell < T_{\wt{B}} < \wt{R}_{\ell + 1}\} \,\kappa^\prime, \;1 \le \ell \le r\,,
\\
\ov{\nu} & = \lambda^\prime \,1 \{\wt{D}_{r + 1} < T_{\wt{B}}\} \,\kappa^\prime\,.
\end{split}
\end{equation}

\n
In a similar fashion one sees that
\begin{equation}\label{5.17}
\mbox{$\mu_\ell, \ell \ge 1$, are independent Poisson measures under $Q_0$}\,,
\end{equation}
and their respective intensity measures on $\cT_{\wt{C}}$ are
\begin{equation}\label{5.18}
\nu_\ell = \lambda \,1\{\wt{D}_\ell < T_{\wt{C}} < \wt{R}_{\ell + 1}\} \,\kappa, \;\ell \ge 1 \,.
\end{equation}
We then define
\begin{equation}\label{5.19}
\cI^* = \bigcup\limits_{1 \le \ell \le r} \;\Big(\bigcup\limits_{w \in {\rm Supp} \,\mu^\prime_\ell} \,{\rm range} \,w \cap A\Big), \;\ov{\cI} = \bigcup\limits_{w \in {\rm Supp} \,\ov{\mu}} \,{\rm range} \,w \cap A\,,
\end{equation}
so that
\begin{equation}\label{5.20}
\mbox{$\cI^\prime = \cI^* \cup \ov{\cI}$, and $\cI^*, \ov{\cI}$ are independent under $Q_3$}\,.
\end{equation}
Note as well that $Q_0$-a.s.,
\begin{equation}\label{5.21}
\cI \cap A = \bigcup\limits_{\ell \ge 1} \;\Big(\bigcup\limits_{w \in {\rm Supp} \,\mu_\ell} \,{\rm range} \,w \cap A\Big)\,.
\end{equation}

\n
The next lemma deviates from the proof of Theorem 3.1 of \cite{Szni08b}, as a consequence of the fact that we work here with simple random walk on $E$ in place of simple random walk on $\IZ^{d+1}$.

\begin{lemma}\label{lem5.2} $(N \ge c(\ve))$
\begin{equation}\label{5.22}
\sup\limits_{x \in \partial C} \;P_x [H_A < T_{\wt{B}}] \le c(\log N)^2 \,(M N^{-\ve})^{d-1}\,.
\end{equation}
\end{lemma}

\begin{proof}
Note that with (\ref{5.10}),
\begin{equation*}
\partial C \subseteq \ov{S} \stackrel{\rm def}{=} S\Big(0, \Big[\mbox{\f $\dis\frac{N}{4M}$}\Big] + 1\Big) \,.
\end{equation*}

\n
The probability that the walk starting in $\ov{S}$ reaches $B(0,\frac{1}{2} \;[\frac{N}{4M}])$ before hitting $\ov{S}$ and then enters $A$ before entering $\ov{S}$, using standard estimates on the one-dimensional walk and on the Green function, cf.~\cite{Lawl91}, p.~31, combined with the right-hand of (\ref{1.7}), satisfies
\begin{equation}\label{5.23}
\sup\limits_{x \in \ov{S}} \;P_x [H_A < \wt{H}_{\ov{S}} \wedge T_{\wt{B}}]  \le c\;\mbox{\f $\dis\frac{M}{N}$} \; \Big(\mbox{\f $\dis\frac{N}{M}$} / N^{1-\ve}\Big)^{-(d-1)} = c \;\mbox{\f $\dis\frac{M}{N}$} \;(M N^{-\ve})^{d-1}\,.
\end{equation}

\n
On the other hand using estimates on the one-dimensional simple random walk to bound from below the probability to move at distance $[c \,\frac{N}{M}]$ of $C \cup \ov{S} = B(0, [\frac{N}{4M}] + 1)$ without hitting $\ov{S}$, then estimates on the Green function together with the right-hand inequality of (\ref{1.7}) to bound from below the probability to reach $\partial B(0,\frac{N}{4})$ without entering $\ov{S}$ and the invariance principle to reach $\IT \times \{[\frac{N}{4}] + N\}$ without entering $\ov{S}$, and then estimates on the simple random walk to bound from below the probability to reach $\IT \times \{h_N\}$ before level $[\frac{N}{4}]$, we see that for $N \ge c(\ve)$:
\begin{equation}\label{5.24}
\inf\limits_{x \in \ov{S}} \;P_x [T_{\wt{B}} < \wt{H}_{\ov{S}} \wedge H_A] \ge c\,\mbox{\f $\dis\frac{M}{N}$} \times c \times \mbox{\f $\dis\frac{N}{h_N - [\frac{N}{4}]}$} \ge c \,\mbox{\f $\dis\frac{M}{N}$} \;(\log N)^{-2} \,.
\end{equation}

\n
With the same argument as in (4.20) of \cite{Szni08b}, (\ref{5.22}) follows from (\ref{5.23}) and (\ref{5.24}).
\end{proof}

We now resume the proof of Proposition \ref{prop5.1} and assume $N \ge c(\alpha, v, \ve)$ so that the tacit assumption above (\ref{5.11}) as well as (\ref{5.22}) hold. We can now bound the total mass of $\ov{\nu}$ in (\ref{5.16}) with the help of the strong Markov property as follows:
\begin{equation}\label{5.25}
\ov{\nu} (\cT_{\wt{B}}) = \lambda^\prime \,P_{e_{A,\wt{B}}} [\wt{D}_{r+1} < T_{\wt{B}}] \underset{(\ref{4.2})}{\stackrel{\rm strong \;Markov}{\le}} c \,\alpha \;{\rm cap}_{\wt{B}}(A) \big(\sup\limits_{x \in \partial C} \,P_x [H_A < T_{\wt{B}}]\big)^r\,.
\end{equation}

\n
Since $\wt{C} \subseteq \wt{B}$, it follows with (\ref{1.3}), (\ref{1.4}) that ${\rm cap}_{\wt{B}}(A) \le {\rm cap}_{\wt{C}}(A)$. Moreover for $N \ge c(\ve)$, as a result of the right-hand inequality of (\ref{1.7}) and standard bounds on the Green function $\sup_{x \in \partial \wt{C}} P_x^{\IZ^{d+1}} [H_A < \infty] \le \frac{1}{2}$. It thus follows that
\begin{equation}\label{5.26}
{\rm cap}_{\wt{C}}(A) - {\rm cap}(A) \stackrel{(\ref{1.3}),(\ref{1.4})}{\le} {\rm cap}_{\wt{C}}(A) \;\sup\limits_{\partial \wt{C}} \;P_x [H_A < \infty]\,,
\end{equation}

\n
whence with standard estimates on the capacity of $A$, cf.~(2.16), p.~53 of \cite{Lawl91}, we find:
\begin{equation}\label{5.27}
{\rm cap}_{\wt{B}}(A) \le {\rm cap}_{\wt{C}}(A)  \le 2 \,{\rm cap}(A) \le c \,N^{(d-1)(1-\ve)}\,.
\end{equation}

\n
Coming back to (\ref{5.25}) we see with (\ref{5.22}) that:
\begin{equation}\label{5.28}
\begin{split}
\ov{\nu} (\cT_{\wt{B}}) & \le c(\alpha) \,N^{(d-1)(1-\ve)} (c (\log N)^2 \;M^{d-1} \;N^{-\ve(d-1)})^r 
\\
&\hspace{-2ex} \, \stackrel{N \ge c(\ve)}{\le} c(\alpha)\,N^{(d-1)(1-\ve) - \frac{3}{4} \,(d-1) \ve r}
  \stackrel{(\ref{5.12})}{\le} c(\alpha) \,N^{-11(d-1)} \stackrel{d\ge 2}{\le} c(\alpha) \,N^{-5d}\,.
\end{split}
\end{equation}
As a result we find that
\begin{equation}\label{5.29}
Q_3 [\ov{\cI} \not= \phi] \le Q_3 [\ov{\mu} \not= 0] \stackrel{(\ref{5.28})}{\le} c\,N^{-3d}, \;\mbox{for} \;N \ge c(\alpha, v,\ve)\,.
\end{equation}

\medskip\n
Then for $\ell \ge 1$, we introduce the map $\phi^\prime_\ell$ from $\{\wt{D}_\ell < T_{\wt{B}} < \wt{R}_{\ell + 1}\} \subseteq \cT_{\wt{B}}$ into $W^{\times \ell}_f$, where $W_f$ denotes the countable collection of finite nearest neighbor paths with values in $\wt{C} \cup \partial \wt{C}$, as well as the map $\phi_\ell$ from $\{\wt{D}_\ell < T_{\wt{C}} < \wt{R}_{\ell + 1}\} \subseteq \cT_{\wt{C}}$ into $W^{\times \ell}_f$ defined by:
\begin{equation}\label{5.30}
\begin{split}
\phi^\prime_\ell (w) & =  \big(w(\wt{R}_k + \cdot)_{0 \le \cdot \le \wt{D}_k - \wt{R}_k}\big)_{1 \le k \le \ell}, \;\mbox{for} \;w \in \{\wt{D}_\ell < T_{\wt{B}} < \wt{R}_{\ell + 1}\} \,,
\\[1ex]
\phi_\ell (w) & =  \big(w(\wt{R}_k + \cdot)_{0 \le \cdot \le \wt{D}_k - \wt{R}_k}\big)_{1 \le k \le \ell}, \;\mbox{for} \;w \in \{\wt{D}_\ell < T_{\wt{C}} < \wt{R}_{\ell + 1}\} \,.
\end{split}
\end{equation}

\n
We can respectively view $\mu^\prime_\ell$ and $\mu_\ell$ for $\ell \ge 1$, as Poisson point processes on $\{\wt{D}_\ell < T_{\wt{B}} < \wt{R}_{\ell + 1}\}$ and $\{\wt{D}_\ell < T_{\wt{C}} < \wt{R}_{\ell + 1}\}$. If $\rho^\prime_\ell$ and $\rho_\ell$ denote their respective images under $\phi^\prime_\ell$ and $\phi_\ell$, we see that with (\ref{5.15}) - (\ref{5.18}),
\begin{align}
&\mbox{$\rho^\prime_\ell$, $1\le \ell \le r$, and $\ov{\mu}$ are independent Poisson point processes,} \label{5.31}
\\[1ex]
&\mbox{$\rho_\ell$, $1 \le \ell$, are independent Poisson point processes,}\label{5.32}
\end{align}

\n
and denoting by $\xi^\prime_\ell$ and $\xi_\ell$ the intensity measures on $W_f^{\times \ell}$ of $\rho^\prime_\ell$ and $\rho_\ell$, we have:
\begin{align}
&\xi^\prime_\ell (dw_1,\dots ,dw_\ell) = \lambda^\prime P_{e_{A,\wt{B}}} [\wt{D}_\ell < T_{\wt{B}} < \wt{R}_{\ell + 1}, (X_{\wt{R}_k + \cdot})_{0 \le \cdot \le \wt{D}_k - \wt{R}_k} \in dw_k, 1 \le k \le \ell]\label{5.33}
\\[1ex]
&\xi_\ell (dw_1,\dots ,dw_\ell) = \lambda^\prime P_{e_{A,\wt{B}}} [\wt{D}_\ell < T_{\wt{C}} < \wt{R}_{\ell + 1}, (X_{\wt{R}_k + \cdot})_{0 \le \cdot \le \wt{D}_k - \wt{R}_k} \in dw_k, 1 \le k \le \ell]\,.\label{5.34}
\end{align}

\n
The following lemma corresponds in the present context to Lemma 3.2 of \cite{Szni08b}. It will be used when comparing $\xi^\prime_\ell$ and $\xi_\ell$, see (\ref{5.37}) below.

\begin{lemma}\label{lem5.3} $(N \ge c(\ve))$

\medskip
For $x \in \partial C$ and $y \in \partial_{\rm int} A$,
\begin{equation}\label{5.35}
P_x [T_{\wt{C}} < \wt{R}_1 < T_{\wt{B}}, X_{\wt{R}_1} = y] \le c_2 \;\dis\frac{(\log N)^2}{M^{d-1}} \;P_x [\wt{R}_1 < T_{\wt{C}}, X_{R_1} = y] \,.
\end{equation}
\end{lemma}

\begin{proof}
We implicitly assume (\ref{5.11}). The same argument leading to (\ref{5.22}), see (\ref{5.23}), (\ref{5.24}), and see also (4.17), (4.18) and (4.20) of \cite{Szni08b}, now yields:
\begin{equation}\label{5.36}
\sup\limits_{x \in \partial \wt{C}} \;P_x [H_{\partial C} < T_{\wt{B}}] \le \mbox{\f $\dis\frac{c(\log N)^2}{M^{d-1}}$} \;.
\end{equation}
Now for $y \in \partial_{\rm int} A$ we find that
\begin{equation*}
\begin{array}{l}
\sup\limits_{z \in \partial C} \;P_z [T_{\wt{C}} < \wt{R}_1 < T_{\wt{B}}, X_{\wt{R}_1} = y] \le \sup\limits_{z \in \partial C} \;E_z \big[P_{X_{T_{\wt{C}}}} [X_{\wt{R}_1} < T_{\wt{B}}, X_{\wt{R}_1} = y]\big] \le
\\[1ex]
\sup\limits_{z^\prime \in \partial \wt{C}} \;P_{z^\prime} [H_{\partial C}  < T_{\wt{B}}] \; \sup\limits_{z \in \partial C} \;P_z[\wt{R}_1 < T_{\wt{B}}, X_{\wt{R}_1} = y] \stackrel{(\ref{5.36})}{\le} \mbox{\f $\dis\frac{c(\log N)^2}{M^{d-1}}$} \;
\\[2ex]
 \sup\limits_{z \in \partial C} \;P_z [\wt{R}_1 < T_{\wt{B}}, X_{\wt{R}_1} = y]\,.
\end{array}
\end{equation*}

\n
Observe that the function $z \rightarrow P_z [\wt{R}_1 < T_{\wt{B}}, X_{\wt{R}_1} = y] = P_z [H_A < T_{\wt{B}}, X_{H_A} = y]$ is harmonic and positive on $\wt{B} \backslash A$ and hence on $\wt{C} \backslash A$ as well. Note that $\wt{C} \cup \partial \wt{C}$ can be identified with a subset of $\IZ^{d+1}$. With the Harnack inequality, cf.~\cite{Lawl91}, p.~42, and a standard covering argument we find that:
\begin{align*}
&\sup\limits_{z \in \partial C} \;P_z [ \wt{R}_1 < T_{\wt{B}}, X_{\wt{R}_1} = y] \le c \;\inf\limits_{z \in \partial C} \; P_z  [\wt{R}_1 < T_{\wt{B}}, X_{\wt{R}_1} = y] \,,
\intertext{and therefore}
&\sup\limits_{z \in \partial C} \;P_z  [ T_{\wt{C}} < \wt{R}_1  < T_{\wt{B}}, X_{\wt{R}_1} = y]  \le c^\prime \;\mbox{\f $\dis\frac{(\log N)^2}{M^{d-1}}$} \;\inf\limits_{z \in \partial C} \;P_z [\wt{R}_1 < T_{\wt{B}}, X_{\wt{R}_1} = y] =
\\[1ex]
&c^\prime \;\mbox{\f $\dis\frac{(\log N)^2}{M^{d-1}}$} \;\inf\limits_{z \in \partial C} \;\big(P_z [T_{\wt{C}} < \wt{R}_1 < T_{\wt{B}}, X_{\wt{R}_1} = y] + P_z [\wt{R}_1 < T_{\wt{C}}, X_{\wt{R}_1} = y]\big)\,.
\end{align*}

\n
Assuming $N \ge c(\ve)$ so that $c^\prime(\log N)^2 \;M^{-(d-1)} \le \frac{1}{2}$, we find that for $x \in \partial C$ and $y \in \partial_{\rm int} A$:
\begin{equation*}
P_x[T_{\wt{C}} < \wt{R}_1 < T_{\wt{B}}, X_{\wt{R}_1} = y] \le 2 c^\prime \;\mbox{\f $\dis\frac{(\log N)^2}{M^{d-1}}$}  \;P_x [\wt{R}_1 < T_{\wt{C}}, X_{\wt{R}_1} = y]\,,
\end{equation*}
and this proves (\ref{5.35}).
\end{proof}

We now continue the proof of Proposition \ref{prop5.1} and will show that for $N \ge c(\alpha, v,\ve)$, 
\begin{equation}\label{5.37}
\xi^\prime_\ell \le \mbox{\f $\dis\frac{\lambda^\prime}{\lambda}$} \;\Big( 1 + c_2 \mbox{\f $\dis\frac{(\log N)^2}{M^{d-1}}$}\Big)^{\ell -1} \xi_\ell, \; \mbox{ for $1 \le \ell \le r$}\,,
\end{equation}

\medskip\n
where we refer to (\ref{4.2}), (\ref{5.4}), (\ref{5.33}) and (\ref{5.34}) for the notation.

\medskip
Given $w \in W_f$, we write $w^s$ and $w^\ell$ for the respective starting point and endpoint of $w$. When $w_1,\dots,w_\ell \in W_f$ we have
\begin{equation}\label{5.38}
\begin{array}{l}
\xi^\prime_\ell \big((w_1,\dots,w_\ell)\big) \stackrel{(\ref{5.33})}{=} 
\\[1ex]
\lambda^\prime \,P_{e_{A,\wt{B}}} [\wt{D}_\ell < T_{\wt{B}} < \wt{R}_{\ell + 1}, (X_{\wt{R}_k + \cdot})_{0 \le \cdot \le \wt{D}_k - \wt{R}_k} = w_k(\cdot), 1 \le k \le \ell] =
\\[1ex]
\dsl_{B \subseteq \{1,\dots,\ell -1\}} \lambda^\prime \,P_{e_{A, \wt{B}}} \big[\wt{D}_\ell < T_{\wt{B}} < \wt{R}_{\ell + 1}, (X_{\wt{R}_k + \cdot})_{0 \le \cdot \le \wt{D}_k - \wt{R}_k}, 1 \le k \le \ell, \;\mbox{and}
\\
\hspace{3.5cm}  \mbox{$T_{\wt{C}} \circ \theta_{\wt{D}_k} + \wt{D}_k < \wt{R}_{k+1}$, exactly for $k \in B$ when}, 
\\[1ex]
\hspace{3.5cm} \mbox{ $1 \le k \le \ell -1\big]$}\,.
\end{array}
\end{equation}

\medskip\n
The above expression vanishes unless $w^s_k \in \partial_{\rm int} A$ and $w^e_k \in \partial C$ and $w_k$ takes values in $C$ except for the final point $w_k^e$, for $1 \le k \le \ell$. If these conditions are satisfied, applying the strong Markov property repeatedly at times $\wt{D}_\ell, \wt{R}_\ell, \wt{D}_{\ell - 1}, \wt{R}_{\ell - 1}, \dots \wt{D}_1$ we find that the last member of (\ref{5.38}) equals:
\begin{equation*}
\begin{array}{l}
\dsl_{B \subseteq \{1,\dots, \ell -1\}} \lambda^\prime \,P_{e_{A,\wt{B}}}  [(X_\point)_{0 \le \cdot \le \wt{D}_1} = w_1(\cdot)] 
\\[1ex]
E_{w_1^e} \big[1\{1 \notin B\} \,1\{T_{\wt{C}} > \wt{R}_1\} + 1\{1 \in B\} \,1\{T_{\wt{C}} < \wt{R}_1\}\,, \wt{R}_1 < T_{\wt{B}}, X_{\wt{R}_1} = w_2^s\big] 
\\[1ex]
P_{w^s_2} [(X_\point)_{0\le \cdot \le \wt{D}_1} = w_2(\cdot)] \dots 
\\[1ex]
E_{w^s_{\ell - 1}} [1\{ \ell - 1\notin B\} \,1\{T_{\wt{C}} > \wt{R}_1\} + 1\{\ell -1 \in B\} \,1\{T_{\wt{C}} < \wt{R}_1\}, \,\wt{R}_1 < T_{\wt{B}}, X_{\wt{R}_1} = w^s_\ell]
\\[1ex]
P_{w^s_\ell} [(X_\point)_{0 \le \cdot \le \wt{D}_1} = w_\ell(\cdot)] P_{w^e_\ell} [T_{\wt{B}} < \wt{R}_1] 
\\[1ex]
\stackrel{(\ref{5.35})}{\le} \dsl_{B \subseteq \{1,\dots, \ell -1\}} \Big(c_2 \;\mbox{\f $\dis\frac{(\log N)^2}{M^{d-1}}$}\Big)^{|B|} \,\lambda^\prime \,P_{e_{A,\wt{B}}}  [(X_\point)_{0 \le \cdot \le \wt{D}_1} =  w_1(\cdot)]
\\[2ex]
P_{w^e_1} [\wt{R}_1 < T_{\wt{C}}, X_{\wt{R}_1} = w^s_2] \,P_{w^s_2} [(X_\point)_{0 \le \cdot \le \wt{D}_1} = w_2(\cdot)] \dots
\\[1ex]
P_{w^e_{\ell -1}}[\wt{R}_1 < T_{\wt{C}}, X_{\wt{R}_1} = w_\ell^s]\, P_{w_\ell^s} [(X_\point)_{0 \le \cdot \le \wt{D}_1} = w_\ell(\cdot)] \, P_{w_\ell^e}[T_{\wt{B}} < \wt{R}_1]
\end{array}
\end{equation*}

\medskip\n
and using the strong Markov property this equals
\begin{equation*}
\begin{array}{l}
\lambda^\prime \Big(1 + c_2 \;\mbox{\f $\dis\frac{(\log N)^2}{M^{d-1}}$}\Big)^{\ell -1} \,P_{e_{A, \wt{B}}} [T_{\wt{C}} \circ \theta_{\wt{D}_k} + \wt{D}_k > \wt{R}_{k+1}, \;\mbox{for $1 \le k \le \ell -1$, $(X_{\wt{R}_k + \cdot})_{0 \le \cdot \le \wt{D}_k - \wt{R}_k}$ }
\\
\hspace{4cm}\mbox{$= w_k(\cdot)$, for $1 \le k \le \ell$}, \;\wt{D}_\ell < T_{\wt{B}} < \wt{R}_{\ell + 1}] \le
\\[1ex]
\lambda^\prime \Big(1+ c_2 \;\mbox{\f $\dis\frac{(\log N)^2}{M^{d-1}}$}\Big)^{\ell -1} \,P_{e_{A, \wt{B}}} [\wt{D}_\ell < T_{\wt{C}} <\wt{R}_{\ell + 1}, (X_{\wt{R}_k + \, \cdot})_{0 \le \cdot \le \wt{D}_k - \wt{R}_k} = w_k(\cdot), \;\mbox{for} \;1 \le k \le \ell] =
\\[2ex]
\mbox{\f $\dis\frac{\lambda^\prime}{\lambda}$} \; \Big(1 + c_2 \;\mbox{\f $\dis\frac{(\log N)^2}{M^{d-1}}$}\Big)^{\ell -1} \;\xi_\ell \big((w_1,\dots,w_\ell)\big)\,.
\end{array}
\end{equation*}

\n
This concludes the proof of (\ref{5.37}).

\medskip
We now assume that $N \ge c(\alpha, v,\ve)$, so that, cf.~(\ref{4.2}), (\ref{5.4}), and (\ref{5.10}):
\begin{equation}\label{5.39}
\lambda \ge \lambda^\prime \;\exp\Big\{ \mbox{\f $\dis\frac{16}{\ve}$} \;c_2 \; \mbox{\f $\dis\frac{(\log N)^2}{M^{d-1}}$}\Big\} \stackrel{(\ref{5.12})}{\ge} \lambda^\prime \Big( 1 + c_2 \; \mbox{\f $\dis\frac{(\log N)^2}{M^{d-1}}$}\Big)^{\ell -1}, \;\mbox{for} \;1 \le \ell \le r \,.
\end{equation}
With (\ref{5.37}) it thus follows that:
\begin{equation}\label{5.40}
\xi^\prime_\ell \le \xi_\ell, \;\mbox{for} \;1 \le \ell \le r \,.
\end{equation}

\n
As a result we find with (\ref{5.19}), (\ref{5.21}), and (\ref{5.30}) that
\begin{equation*}
\cI^* = \bigcup\limits_{1 \le \ell \le r} \;\bigcup\limits_{(w_1,\dots,w_\ell) \in {\rm Supp}\,\rho_\ell^\prime} ({\rm range} \,w_1 \cup \dots \cup {\rm range} \,w_\ell) \cap A
\end{equation*}
and that
\begin{equation*}
\cI\cap A \supseteq \bigcup\limits_{1 \le \ell \le r} \;\bigcup\limits_{(w_1,\dots,w_\ell) \in {\rm Supp}\,\rho_\ell} ({\rm range} \,w_1 \cup \dots \cup {\rm range} \,w_\ell) \cap A\,.
\end{equation*}

\n
It then follows from (\ref{5.31}), (\ref{5.32}), and (\ref{5.40}) that
\begin{equation}\label{5.41}
\mbox{$\cI \cap A$ under $Q_0$ stochastically dominates $\cI^*$ under $Q_3$} \,.
\end{equation}

\n
Combining (\ref{5.20}), (\ref{5.29}), and (\ref{5.41}) we have proved Proposition \ref{prop5.1}.
\end{proof}

We are now ready to construct the coupling of $X_\point, X_\point^k, k \ge 1, \cI^\prime$ under $Q_3$ with $\mu$ and $\cI$ under $Q_0$ as announced at the beginning of this section.

\begin{proposition}\label{prop5.4} $(\alpha > 0, v > (d + 1)\,\alpha, 0 < \ve < 1)$

\medskip
For $N\ge c(\alpha, v, \ve)$ and $z_0 \in I$, one can construct on an auxiliary space $(\Omega_4, \cA_4)$ a coupling $Q_4$ of $X_\point, X_\point^{\prime k}$, $k \ge 1$, $\cI^\prime$ under $Q_3$ with $\mu$, $\cI$ under $Q_0$ so that
\begin{equation}\label{5.42}
Q_4 [X_{[0,D_K]} \cap A \subseteq \cI] \ge 1 - c\,N^{-3d} \,.
\end{equation}
\end{proposition}

\begin{proof}

With $N \ge c(\alpha, v,\ve)$ as in Proposition \ref{prop5.1}, we chose $\Omega_4 = \Omega_3 \times \Omega_0$, $\cA_4 = \cA_3 \otimes \cA_0$, and consider the conditional probabilities for $B^*, B \subseteq A$:
\begin{align*}
&Q_3 [\cdot \,| \,\cI^* = B^*],  \;\mbox{understood as $Q_3$ when $Q_3[\cI^* = B^*] = 0$}\,,
\\[1ex]
&Q_0 [\cdot \,| \,\cI \cap A = B],  \;\mbox{understood as $Q_0$ when $Q_0[\cI \cap A = B] = 0$}\,.
\end{align*}

\n
Letting $\cP(A)$ stand for the collection of subsets of $A$, we can construct with (\ref{5.9}) and Theorem 2.4, p.~72 of \cite{Ligg85}, a probability $p$ on $\cP(A)^2$ coupling the distribution of $\cI^*$ under $Q_3$ and that of $\cI \cap A$ under $Q_0$, such that $p$-a.s., the first coordinate on $\cP(A)^2$, (which is distributed as $\cI^*$ under $Q_3)$, is a subset of the second coordinate, (which is distributed as $\cI \cap A$ under $Q_0$). We then define on $\Omega_4 = \Omega_3 \times \Omega_0$ the probability
\begin{equation}\label{5.43}
Q_4 [\cdot ] = \dsl_{B^*, B \subseteq A} \,p(B^*,B) \;Q_3 [\cdot \,| \,\cI^* = B^*] \otimes Q_0 [\cdot \,|\,\cI \cap A = B]\,. 
\end{equation}
This probability yields a coupling of $X_\point, X_\point^{\prime k}$, $k \ge 1$, $\cI^\prime$ under $Q_3$ with $\mu$, $\cI$ under $Q_0$. Moreover in view of (\ref{5.8}) and (\ref{5.6}) we find that
\begin{equation}\label{5.44}
\begin{split}
Q_4 [\cI \cap A \supseteq \cI^\prime] & \stackrel{(\ref{5.6})}{\ge} Q_4 [\cI \cap A \supseteq \cI^*, \;\ov{\cI} = \phi] \stackrel{(\ref{5.43})}{=} Q_4 [\ov{\cI} = \phi] 
\\
& \stackrel{(\ref{5.9})}{\ge} 1 - c\,N^{-3d} \,.
\end{split}
\end{equation}
Together with (\ref{4.4}) this yields (\ref{5.42}). 
\end{proof}

\section{Comparison with random interlacements}
\setcounter{equation}{0}

In this section we complete the proof of Theorem \ref{theo1.1}, (cf.~step f) of the outline following Theorem \ref{theo1.1}). We can view $\wt{C} \cup \partial \wt{C}$ as a subset of $\IZ^{d+1}$, and the main ingredient is to stochastically dominate $\cI \cap A$, which is the trace on $A$ of the ranges of trajectories in the support of the Poisson point measure $\mu$ on $\cT_{\wt{C}}$ with intensity measure $\lambda \kappa$, cf.~(\ref{5.4}), with the trace on $A$ of random interlacements at level $v$. In view of (\ref{1.18}), (\ref{1.19}) it suffices for this purpose to dominate the equilibrium measure $e_{A, \wt{B}}$ which appears in (\ref{5.4}), with a multiple slightly bigger than $1$ of the equilibrium measure $e_A$ of $A$ relative to $\IZ^{d+1}$. This is carried out in Proposition \ref{prop6.1}.

\begin{proposition}\label{prop6.1} $(\alpha > 0, v > (d+1) \alpha, 0 < \ve < 1)$

\medskip
For $N \ge c(\alpha, v,\ve)$ and $z_0 \in I$,
\begin{equation}\label{6.1}
\mbox{$\cI \cap A$ under $Q_4$ is stochastically dominated by $\cI^v \cap A$ under $\IP$}\,.
\end{equation}
\end{proposition}

\begin{proof}
The random set $\cI \cap A$ is the trace on $A$ of the ranges of trajectories in the support of the Poisson point measure $\mu$ on $\cT_{\wt{C}}$ with intensity measure, cf.~(\ref{5.4}), 
\begin{equation*}
\lambda \kappa (dw) = \Big(1 + \mbox{\f $\dis\frac{4}{5}$}\,\delta \Big) \,\alpha(d + 1) \Big( 1 - \mbox{\f $\dis\frac{r_N}{h_N}$}\Big) \,P_{e_{A, \wt{B}}} [X_{\cdot \wedge T_{\wt{C}}} \in dw] \,.
\end{equation*}

\n
On the other hand $\cI^v \cap A$ is the trace on $A$ of the ranges of trajectories in the support of the Poisson point measure $\mu_{A,v}$ on $W_+$ with intensity measure, cf.~(\ref{1.19}):
\begin{equation*}
v \,P_{e_A} [X_\point \in dw]\,.
\end{equation*}

\n
The claim (\ref{6.1}) will thus follow as soon as we show that for $N \ge c(\alpha, v,\ve)$, 
\begin{equation}\label{6.2}
\Big(1 + \mbox{\f $\dis\frac{4}{5}$}\,\delta\Big) \,\alpha (d + 1) \Big(1 - \mbox{\f $\dis\frac{r_N}{h_N}$}\Big) \,e_{A,\wt{B}} \le ve_A\,.
\end{equation}

\n
To this end we note with similar arguments as above (\ref{5.26}), that one has for $x \in \partial_{\rm int} A$:
\begin{equation}\label{6.3}
\begin{split}
e_{A, \wt{B}}(x) - e_A(x) &\stackrel{\wt{C} \subseteq \wt{B}}{\le} e_{A, \wt{C}}(x) - e_A(x) \stackrel{(\ref{1.3})}{\le} P^{\IZ^{d+1}}_x [T_{\wt{C}} < \wt{H}_A < \infty]
\\[1ex]
&\hspace{-3ex}\stackrel{\rm strong \;Markov}{\le} e_{A, \wt{C}}(x) \;\sup\limits_{x^\prime \in \partial \wt{C}} P_{x^\prime}^{\IZ^{d+1}} [H_A < \infty] \le e_{A,\wt{C}}(x) \,c\,N^{-\ve(d-1)}\,,
\end{split}
\end{equation}

\n
using the right-hand inequality of (\ref{1.7}) and standard bounds on the Green function, cf.~\cite{Lawl91}, p.~31. We thus find that for $N \ge c(\ve)$,
\begin{equation}\label{6.4}
\begin{array}{l}
e_{A,\wt{B}}(x) \le e_{A,\wt{C}}(x) \le e_A(x) (1 - c\,N^{-\ve(d-1)})^{-1} \le 
\\[1ex]
e_A(x) (1 + c^\prime \,N^{-\ve(d-1)}), \;\mbox{for all} \;x \in \partial_{\rm int} A\,.
\end{array}
\end{equation}

\n
This is more than enough to show that (\ref{6.2}) holds and this concludes the proof of Proposition \ref{prop6.1}.
\end{proof}

We now turn to the

\medskip\n
{\it Proof of Theorem \ref{theo1.1}:} We assume $N \ge c(\alpha,v,\ve)$ and $z_0 \in I$ as in Proposition \ref{prop6.1}. We consider the space $\Omega^\prime = \Omega_4 \times \Omega$, cf.~(\ref{1.14}), endowed with the product $\sigma$-algebra $\cA^\prime = \cA_4 \otimes \cA$. We endow $(\Omega^\prime, \cA^\prime)$ with a probability $Q^\prime$ as follows. Using a similar construction as in (\ref{5.43}) we consider a probability $p^\prime$ on $\cP(A)^2$ coupling the law of $\cI \cap A$ under $Q_4$ with the law of $\cI^v \cap A$ under $\IP$, such that $p^\prime$-a.s. the first coordinate is a subset of the second coordinate. We then define the probability $Q^\prime$ on $(\Omega^\prime, Q^\prime)$ via
\begin{equation}\label{6.5}
Q^\prime [ \cdot] = \dsl_{A_1,A_2 \subseteq A} p^\prime(A_1,A_2) \,Q_4 [ \cdot \,| \,\cI \cap A = A_1] \otimes \IP [ \cdot \,| \,\cI^v \cap A = A_2]\,,
\end{equation}

\n
where we use a similar convention as below (\ref{5.42}) to define the conditional probabilities appearing in (\ref{6.4}) when either $Q_4[\cI \cap A = A_1]$ or $\IP[\cI^\nu \cap A = A_2]$ vanishes. As a result of (\ref{5.42}) we thus find that:
\begin{equation}\label{6.6}
Q^\prime [X_{[0,D_K]} \cap A \subseteq \cI^v \cap A] \ge 1 - c\,N^{-3d}\,.
\end{equation}

\n
The coupling $Q^\prime$ satisfies the estimate (\ref{2.1}) and enables with Proposition \ref{prop2.1} to complete the proof of Theorem \ref{theo1.1}. \hfill $\square$

\section{Lower bound on the disconnection time}
\setcounter{equation}{0}

In this section we apply Theorem \ref{theo1.1} together with the controls of \cite{SidoSzni09a} recalled in (\ref{1.23}) to prove a lower bound on the disconnection time $T_N$ of the discrete cylinder, see (\ref{7.1}) for the definition of $T_N$. We derive in Theorem \ref{theo7.3} a lower bound on $T_N$, which in particular shows that under $P$ the laws of $N^{2d}/T_N$, $N \ge 2$, are tight when $d \ge 2$. This had previously only been proved when $d \ge 17$, cf.~\cite{DembSzni08}. Together with Corollary 4.6 of \cite{Szni08b} this shows that for all $d \ge 2$, ``$T_N$ lives in scale $N^{2d}$''. An additional interest of Theorem \ref{theo1.1} stems from the fact that better controls on the percolative properties of the vacant set of random interlacements $\cV^u$ when $u < u_*$, should lead to an improvement of the lower bound on $T_N$ we derive here, cf.~Remark \ref{rem7.5} 2).

\medskip
We begin with some terminology and notation. A finite subset $S$ of $E$, cf.~(\ref{0.1}), is said to disconnect $E$ when for large $M$, $E \times (-\infty, -M]$ and $E \times [M, \infty)$ belong to distinct connected components of $E \backslash S$. The disconnection time of $E$ by the simple random walk $X_\point$ is then defined as
\begin{equation}\label{7.1}
T_N = \inf\{ n \ge 0; \;X_{[0,n]} \;\mbox{disconnects}\;  E\}\,.
\end{equation}

\n
It is convenient to introduce the sequence $\rho_m$, $m \ge 0$ of successive displacements of the vertical component $Z_\point$ of $X_\point$:
\begin{equation}\label{7.2}
\mbox{$\rho_0 = 0$, and by induction $\rho_{m+1} = \inf\{k > \rho_m; \;Z_k \not= Z_{\rho_m}\}$, for $m \ge 0$}\,,
\end{equation}

\n
as well as the time changed process and its local time:
\begin{equation}\label{7.3}
\wh{Z}_m = ~Z_{\rho_m}, m \ge 0, \;\mbox{and} \ \wh{L}^z_k = \dsl_{0 \le m < k} \,1\{\wh{Z}_m = z\}, \;\mbox{for $z \in \IZ, k \ge 0$}\,.
\end{equation}

\n
Note that under $P$, (see below (\ref{0.1}) for the notation), $\wh{Z}_\point$ is distributed as a simple random walk on $\IZ$ starting at the origin. We further introduce the random times:
\begin{equation}\label{7.4}
\gamma^z_u = \inf\{\rho_k; \;k \ge 0, \,\wh{L}^z_k \ge u\}, \;\mbox{for $u \ge 0, z \in \IZ$} \,.
\end{equation}

\n
We recall the notation for $K$ below (\ref{1.24}), and for $D^z_k$ in (\ref{1.10}). In the next proposition we will show that $\inf_{z \in \IZ} \,D^z_K$ happens at least in scale $N^{2d}$. More is true, see Remark \ref{rem7.2}, but the controls in Proposition \ref{prop7.1} will be sufficient for our purpose. We let $W$ stand for the canonical Wiener measure and consider, cf.~(\ref{0.4})
\begin{equation}\label{7.5}
\zeta(u) = \inf\{t \ge 0; \;\sup\limits_{a \in \IR} \;L(a,t) \ge u\}, \;\mbox{for $u \ge 0$} \,,
\end{equation}

\n
with $L(v,t)$ a jointly continuous version of the local time of the canonical Brownian motion. The Laplace transform of $\zeta(u)$ is known thanks to the works \cite{Boro84}, and \cite{Eise90}, p.~89. One has the identity, see also (0.12) of \cite{Szni08b}:
\begin{equation}\label{7.6a}
E^W[e^{-\frac{\theta^2}{2} \,\zeta(u)}] = \dis\frac{\theta u}{[\sinh (\frac{\theta u}{2})]^2} \; \; \dis\frac{I_1(\frac{\theta u}{2})}{I_0 (\frac{\theta u}{2})}, \;\mbox{for $\theta, u > 0$}\,.
\end{equation}

\begin{proposition}\label{prop7.1} $(d \ge 2, \alpha > 0)$

\begin{align}
&\lim\limits_N \;N^{2d + 1} \sup\limits_{z \in \IZ} \;P[D^z_K < \gamma_{\alpha^\prime N^d}^z] = 0, \; \mbox{for $0 < \alpha^\prime < \alpha$} \,. \label{7.6}
\\[1ex]
&\underset{N}{\underline{\lim}} \;P \big[\bigcap\limits_{|z| \le N^{2d +1}} \{D^z_K > \gamma \,N^{2d}\}\big] \ge W [\zeta (\sqrt{d + 1} \,\alpha) > \gamma ], \;\mbox{for} \;\gamma > 0 \,. \label{7.7}
\end{align}
\end{proposition}

\begin{proof}
We begin with the proof of (\ref{7.6}) which constitutes an intermediary step in the proof of (\ref{7.7}). Consider $z \in \IZ$, and observe that under any $P_x$, when $\pi_\IZ(x) = z$, the number of visits of $\wh{Z}_\point$ to $z$ before exiting $z + \wt{I}$, see (\ref{1.9}) for the definition of $\wt{I}$, almost surely equals $\sum_{m \ge 0} 1 \{\wh{Z}_m = z, \rho_m < T_{\wt{B}(z)}\}$, and is distributed as a geometric random variable with success probability $h_N^{-1}$. Applying the strong Markov property at the times $R^z_{k^\prime}$, $1 \le k^\prime \le k$, we see that
\begin{equation}\label{7.8}
\begin{array}{l}
\mbox{under $P$, $\sum_{m \ge 0} \;1 \{\wh{Z}_m = z, \rho_m < D^z_K\}$ stochastically dominates the sum of}
\\
\mbox{$K$ independent variables distributed as $UV$, where $U$ is a Bernoulli variable}
\\
\mbox{with success probability $\frac{h_N - r_N}{h_N}$, and $V$ an independent geometric variable of}
\\
\mbox{parameter $h_N^{-1}$, (in fact when $|z| \ge r_N$ there is an equality of distribution).}
\end{array}
\end{equation}

\n
It then follows that for $a> 0$ and $z \in \IZ$, with Chebishev's inequality:
\begin{equation}\label{7.9}
\begin{array}{l}
P[D_K^z < \gamma_{\alpha^\prime N^d}^z] = P\big[\dsl_{m \ge 0} 1 \{\wh{Z}_m = z, \rho_m < D_K^z\} < \alpha^\prime N^d] \le
\\[2ex]
\exp\Big\{ \mbox{\f $\dis\frac{a}{h_N}$} \;\alpha^\prime N^d\Big\} \,E\big[e^{-\frac{a}{h_N} \,UV}\big]^K = 
\\[1ex]
\exp\Big\{ \mbox{\f $\dis\frac{a}{h_N}$} \;\alpha^\prime N^d + K \log \Big( \mbox{\f $\dis\frac{r_N}{h_N}$} +  \mbox{\f $\dis\frac{h_N - r_N}{h_N}$}  \;  \mbox{\f $\dis\frac{e^{-\frac{a}{h_N}}}{h_N}$} \; \mbox{\f $\dis\frac{1}{1 - e^{-\frac{a}{h_N}}(1 - \frac{1}{h_N})}$}\Big\}\,.
\end{array}
\end{equation}

\n
For large $N$ the second term inside the exponential in the last member of (\ref{7.9}) is equivalent to $\alpha \,\frac{N^d}{h_N} \;\log (\frac{1}{1+a})$, and since $\alpha^\prime < \alpha$, the claim (\ref{7.6}) follows from (\ref{7.9}) by choosing $a > 0$ small enough.

\medskip
We now turn to the proof of (\ref{7.7}). With (2.20) of \cite{CsakReve83}, we know that we can construct an auxiliary space $(\ov{\Omega},\ov{\cA},\ov{P})$ coupling $\wh{L}^z_k$, $z \in \IZ$, $k \ge 0$, under $P$ with $L(a,t)$, $a \in \IR$, $t \ge 0$, under $W$ so that
\begin{equation}\label{7.10}
\mbox{$\ov{P}$-a.s., $\sup\limits_{z \in \IZ, k \ge 1} \;\dis\frac{|\wh{L}^z_k - L(z,k)|}{k^{\frac{1}{4} + \eta}} < \infty$, for any $\eta > 0$} \,.
\end{equation}

\n
Note that when $N \ge 3$, the sequence $\rho_m$, $m \ge 0$, under $P$ has a distribution independent of $N$, namely the law of the successive partial sums of independent geometric variables with success probability $(d+1)^{-1}$. Thus for $\gamma^\prime > \gamma > 0$ and $\alpha > \alpha^\prime > \alpha^{\prime\prime} > 0$, we see with (\ref{7.6}) and the law of large number, that
\begin{equation}\label{7.11}
\begin{array}{l}
\underset{N}{\underline{\lim}} \;P \big[\bigcap\limits_{|z| \le N^{2d+1}} \{D^z_K > \gamma N^{2d}\}\big] \stackrel{(\ref{7.6})}{\ge} \underset{N}{\underline{\lim}} \;P \big[\bigcap\limits_{|z| \le N^{2d+1}} \{\gamma^z_{\alpha^\prime N^d} > \gamma \,N^{2d}\}\big] \ge
\\[1ex]
\underset{N}{\underline{\lim}} \;P \big[\bigcap\limits_{|z| \le N^{2d+1}} \big\{\gamma^z_{\alpha^\prime N^d} > \rho_{[\frac{\gamma^\prime}{d+1} \,N^{2d}]}\big\}\big] \stackrel{(\ref{7.4})}{=}
\\[1ex]
 \underset{N}{\underline{\lim}} \;P \big[\bigcap\limits_{|z| \le N^{2d+1}} \big\{\wh{L}^z_{[\frac{\gamma^\prime}{d+1}\,N^{2d}]} < \alpha^\prime N^d\big\}\big] \stackrel{(\ref{7.10})}{\ge}
\\[1ex]
 \underset{N}{\underline{\lim}} \;\ov{P}  \Big[\bigcap\limits_{|z| \le N^{2d+1}}\; \Big\{L \Big( z, \,
\Big[\mbox{\f $\dis\frac{\gamma^\prime}{d+1}$}\,N^{2d}\Big]\Big)  < \alpha^{\prime\prime} \,N^d\Big\}\Big] \stackrel{\rm scaling}{=}
 \\[1ex]
  \underset{N}{\underline{\lim}} \;W  \Big[\bigcap\limits_{|z| \le N^{2d+1}} \Big\{L \Big( \mbox{\f $\dis\frac{z}{N^d}$}, \,\Big[ \mbox{\f $\dis\frac{\gamma^\prime}{d+1}$} \; N^{2d}\Big] / N^{2d}\Big) < \alpha^{\prime\prime}\Big\}\Big] \ge
  \\[1ex]
  W \Big[\sup\limits_{a \in \IR} \;L\Big(a, \dis\frac{\gamma^\prime}{d+1}\Big) < \alpha^{\prime\prime}\Big] =  W\Big[\zeta(\alpha^{\prime\prime}) > \dis\frac{\gamma^\prime}{d+1}\Big] \stackrel{\rm scaling}{=} W[\zeta(1) > \mbox{\f $\dis\frac{\gamma^\prime}{(d+1)\alpha^{\prime\prime 2}}$}\Big]\,.
\end{array}
\end{equation}

\medskip\n
Letting $\gamma^\prime$ decrease to $\gamma, \alpha^{\prime\prime}$ increase to $\alpha$, as well as scaling we find (\ref{7.7}).
\end{proof}

\begin{remark}\label{rem7.2} \rm
Let us mention that one can show in a very similar fashion that
\begin{equation*}
\lim\limits_N \;N^{2d+1} \;\sup\limits_{z \in \IZ} \;P[D^z_K > \gamma^z_{\alpha^\prime N^d} ] = 0, \;\mbox{for} \; \alpha^\prime > \alpha\,.
\end{equation*}

\n
(one simply uses the fact that the sum in the fist line of (\ref{7.8}) is stochastically dominated by the sum of $k$ independent variables distributed as $V$, and a very similar exponential Chebishev inequality as in (\ref{7.9})).

\medskip
More importantly one can also show with similar manipulations as in (\ref{7.11}) that
\begin{equation*}
\overline{\lim\limits_N} \;P \big[\bigcap\limits_{|z| \le N^{2d+1}} \{D^z_K > \gamma \,N^{2d}\}] \le W [\zeta (\sqrt{d+1} \,\alpha) \ge \gamma\big], \;\mbox{for} \;\gamma > 0\,,
\end{equation*}

\n
we also refer to (4.38) of \cite{Szni08b} for a similar calculation.

\medskip
Combined with (\ref{7.7}), this shows that
\begin{equation}\label{7.12}
\mbox{under $P$, $\inf\limits_{|z| \le N^{2d + 1}} \;D^z_K/N^{2d}$ converges in law to $\zeta \,(\sqrt{d+1} \,\alpha)$}\,.
\end{equation}

\n
Note also that with scaling one has the identity
\begin{equation*}
\zeta \,(\sqrt{d+1} \,\alpha) \stackrel{\rm law}{=} (d+1) \zeta(\alpha) = \inf\Big\{t \ge 0; \;\sup\limits_{a \in \IR} (d+1) \,L\Big(a, \mbox{\f $\dis\frac{t}{d+1}$}\Big) \ge (d+1) \alpha\Big\}\,.
\end{equation*}

\n
This last expression has a strong intuitive content in terms of random interlacements attached to random walk on the cylinder $E$. Indeed in view of Theorem 0.1 of \cite{Szni08b}, $(d+1) \,L(a, \frac{t}{d+1})$ corresponds, loosely speaking, to the level of the random interlacement governing for large $N$ the local picture at times of order $t\,N^{2d}$ left by the random walk in the neighborhood of a point with vertical projection of order $a N^d$. On the other hand $(d+1) \,\alpha$ is the level of the random interlacement which naturally shows up in describing the local picture left by the walk near some point $x$ at height $z$ by time $D_K^z$. Incidentally in the same vein as (\ref{7.12}) one can show that for $a$ in $\IR$ and $z_N \sim a\,N^d$, $D_K^{z_N}/N^{2d}$ under $P$ converges in distribution to $\inf\{s \ge 0, L(a,\frac{s}{d+1}) \ge \alpha\}$. \hfill $\square$
\end{remark}

We now come to the main result of this section.

\begin{theorem}\label{theo7.3} $(d \ge 2)$

\medskip
For small $v > 0$,
\begin{equation}\label{7.13}
\mbox{for} \;\gamma > 0, \;\underset{N}{\underline{\lim}} \;P[T_N > \gamma\,N^{2d}] \ge W \Big[\zeta \Big(\mbox{\f $\dis\frac{v}{\sqrt{d+1}}$}\Big) > \gamma\Big] \,,
\end{equation}

\medskip\n
and in particular the laws of $N^{2d} / T_N$, $N \ge 2$ are tight.

\medskip\n
(We refer to Remark {\rm \ref{rem7.5}} for the explanation of why we write $\frac{v}{\sqrt{d+1}}$ for the parameter entering $\zeta(\cdot)$ in {\rm (\ref{7.13})} ).
\end{theorem}

\begin{proof}
We denote with $v_0$ the value $u(\rho = 6d)$ which appears in (\ref{1.23}) and choose $\alpha > 0$, so that
\begin{equation}\label{7.14}
v_0 > (d+1) \,\alpha \,.
\end{equation}

\n
Then for $\gamma > 0$, we can write
\begin{equation*}
\begin{split}
\underset{N}{\underline{\lim}} \;P[T_N > \gamma\,N^{2d}] & \ge \underset{N}{\underline{\lim}} \;P \big[ \inf\limits_{|z| \le N^{2d+1}} D^z_K > \gamma\,N^{2d}\big]  - \overline{\lim\limits_N} \;P \big[ \inf\limits_{|z| \le N^{2d+1}} D^z_K > \gamma\,N^{2d} \ge T_N\big]
\\[1ex]
&\hspace{-1ex} \stackrel{(\ref{7.7})}{\ge} W[ \zeta \,(\sqrt{d+1} \,\alpha) > \gamma] - \overline{\lim\limits_N} \; \big[ \inf\limits_{|z| \le N^{2d+1}} D^z_K > \gamma\,N^{2d} \ge T_N\big]
\end{split}
\end{equation*}

\n
Once we show that for $\gamma > 0$ and $\alpha$ as in (\ref{7.14})
\begin{equation}\label{7.15}
\lim\limits_N \; P \big[ \inf\limits_{|z| \le N^{2d+1}} D^z_K > \gamma\,N^{2d} \ge T_N\big] = 0\,,
\end{equation}

\n
the claim (\ref{7.13}) will follow for any $v < v_0$, (and in fact even for $v = v_0$, using a similar argument as below (\ref{7.11})). To prove (\ref{7.15}) we will rely on 

\begin{lemma}\label{lem7.4} $(N \ge c(\gamma))$

\medskip
$P$-a.s. on $\{T_N < \gamma\,N^{2d}\}$ there exists $x_* = z_* \,e_{d+1}$, with $|z_*| \le N^{2d+1}$, and a $*$-path in $U \cap X_{[0,T_N]}$starting at $x_*$ and ending in $S(x_*, [\sqrt{N}])$, where $U$ is the planar strip
\begin{equation}\label{7.16}
U = \big[ -2 [\sqrt{N}], \;2[\sqrt{N}]\big] \,e_1 + \IZ \,e_{d+1}\,,
\end{equation}

\n
(viewed both as a subset of $\IZ^{d+1}$ and $E$).
\end{lemma}

\begin{proof}
For $N \ge c(\gamma)$, $P$-a.s. on $\{T_N < \gamma\,N^{2d}\}$, $\IT \times (-\infty, - N^{2d+1}]$ and $\IT \times [N^{2d+1} - 1, \infty)$ are in distinct connected components of $E \backslash X_{[0,T_N]}$. Consequently the connected component $O$ in $U$ of $U \backslash X_{[0,T_N]}$ containing $U \cap (\IT \times [N^{2d+1}, \infty))$ does not meet $U \cap (\IT \times (-\infty, - N^{2d+1}])$. Consider $x = z e_{d+1}$ the point of minimal height on $\IZ e_{d+1}$ belonging to $O$, so that $|z| < N^{2d+1}$, and set $x_* \,e_{d+1}$ with $z_* = z-1$.  With Proposition \ref{prop2.1}, p.~29 of \cite{Kest82}, we can find a $*$-loop surrounding the connected set $O^\prime = O \cap (\IT \times [-N^{2d+1}, N^{2d+1}]) \supseteq [-2[\sqrt{N}]$, $2[\sqrt{N}]]\,e_1 + N^{2d+1} e_{d+1}$, contained in $\partial O^\prime \cap (\IZ e_1 + \IZ e_{d+1})$ and passing through $x_*$. However points of $\partial O^\prime$ in $B(x_*, \sqrt{N}) \cap U$ necessarily belong to $X_{[0,T_N]}$. We can thus extract from the $*$-loop a $*$-path from $x_*$ to $S(x_*,[\sqrt{N}])$ contained in $U \cap X_{[0,T_N]}$.
\end{proof}

With the above lemma, the expression in (\ref{7.15}) is smaller than
\begin{equation} \label{7.17}
\begin{array}{l}
\overline{\lim\limits_N} \;P \big[  \mbox{for some $|z_*| \le N^{2d+1}$ there is a $*$-path from $x_* = z_* \,e_{d+1}$}\\
\qquad  \quad \mbox{to $S(x_*, [\sqrt{N}])$ in $U \cap X_{[0,D^{z_*}_K]}\big] \stackrel{\rm Theorem \,\ref{theo1.1}, \,(\ref{7.14})}{\le}$}
\\[3ex]
\overline{\lim\limits_N} \;c \,N^{2d+1} \,(\IP \big[\mbox{there is a $*$-path from $O$ to $S(0,[\sqrt{N}])$}
\\
\qquad \qquad \qquad  \;\, \mbox{in $\cI^{v_0} \cap (\IZ e_1 + \IZ e_{d+1})\big] + N^{-3d}) = 0,$}
\\[1ex]
\mbox{in view of (\ref{1.23}) and our choice of $v_0$}\,.
\end{array}
\end{equation}

\medskip\n
This proves (\ref{7.15}) and concludes the proof of Theorem \ref{theo7.3}. 
\end{proof}

\begin{remark}\label{rem7.5} \rm ~

\medskip\n
1) As already mentioned in the Introduction, cf.~(\ref{0.4}), it follows from Theorem \ref{theo7.3} above and Corollary 4.6 of \cite{Szni08b} that for $d \ge 2$, ``$T_N$ lives in scale $N^{2d}$'', i.e.~more precisely \linebreak under $P$
\begin{equation}\label{7.18}
T_N / N^{2d} \;\mbox{and} \;N^{2d} / T_N, \; N \ge 2, \;\mbox{are tight} \,.
\end{equation}

\n
One can also argue in a direct fashion with the help of the invariance principle that (\ref{7.18}) holds as well when $d = 1$.

\bigskip\n
2) It is an open problem, cf.~Remark 4.7 2) of \cite{Szni08b}, whether for $d \ge 2$, under $P$
\begin{equation}\label{7.19}
\mbox{$T_N / N^{2d}$ converges in law towards $\zeta \Big(\mbox{\f $\dis\frac{u_*}{\sqrt{d+1}}$}\Big)$, as $N \rightarrow \infty$}\,,
\end{equation}

\medskip\n
with $u_*$ the non-degenerate critical value for the percolation of the vacant set of random interlacements, see below (\ref{1.22}). 

\medskip
It has been shown in Corollary 4.6 of \cite{Szni08b} that when $d \ge 2$,
\begin{equation}\label{7.20}
\mbox{for} \;\gamma > 0, \;\overline{\lim\limits_N} \;P[T_N \ge \gamma\,N^{2d}] \le W \Big[  \zeta \Big(\mbox{\f $\dis\frac{u_{**}}{\sqrt{d+1}}$}\Big) \ge \gamma\Big] \,,
\end{equation}

\n
with $u_{**} \in [u_*, \infty)$ a certain critical value introduced in (\ref{0.6}) of \cite{Szni08b}, above which there is a power decay in $L$ of finding a path in $\cV^u$ from $B(0,L)$ to $S(0,2L)$.

\medskip
Showing that $u_* = u_{**}$ and that one can choose $v = u_*$ in (\ref{7.13}) would yield a proof of (\ref{7.19}). One  interest of Theorem \ref{theo1.1} is that this last statement will follow if one can derive some suitable quantitative estimates on the presence of the infinite cluster in $\cV^u$, when $u < u_*$, see also \cite{Teix08c}. In a similar fashion the identity $u_* = u_{**}$ will follow if one can prove quantitative controls on the rarity of large finite clusters in $\cV^u$ when $u > u_*$. \hfill $\square$

\end{remark}

\end{document}